\documentclass{article}
\usepackage[utf8]{inputenc} \usepackage[T1]{fontenc}    \usepackage[english]{babel}

\usepackage{amsfonts}
\usepackage{nicefrac}       \usepackage{microtype}      \usepackage{lmodern}
\usepackage{amssymb,amsmath,amsthm}
\usepackage{stmaryrd}
\usepackage{bbm,bm}
\usepackage{latexsym}
\usepackage{xcolor}

\usepackage{enumerate}
\usepackage{verbatim}
\usepackage{booktabs}       

\usepackage{url}            \usepackage[colorlinks=true]{hyperref}
\usepackage[numbers,sort&compress,square,comma]{natbib}
\usepackage[capitalise,noabbrev]{cleveref}

\usepackage{parskip}
\usepackage{geometry}
\usepackage[short]{optidef}
\usepackage{authblk}

\usepackage{thmtools}

\declaretheorem{theorem}
\declaretheorem{corollary}
\declaretheorem{lemma}
\declaretheorem{proposition}
\declaretheorem{observation}
\declaretheorem{fact}

\declaretheoremstyle[qed=$\square$]{definitionwithend}
\declaretheorem[style=definitionwithend]{definition}
\declaretheorem[style=definitionwithend]{assumption}
\declaretheorem[style=definitionwithend]{example}
\declaretheorem[style=definitionwithend]{remark}

\crefname{fact}{Fact}{Facts}
\crefname{algorithm}{Algorithm}{Algorithms}
\crefname{assumption}{Assumption}{Assumptions}

\usepackage{import}
\usepackage{pdfpages}
\usepackage{transparent}

\definecolor{gold}{rgb}{0.85,0.65,0}

\newcommand{\abs}[1]{\ensuremath{\left\lvert #1 \right\rvert}}

\newcommand{\by}{\times}
 
\newcommand{\norm}[1]{\ensuremath{\left\lVert #1 \right\rVert}}
\newcommand{\ip}[1]{\ensuremath{\left\langle #1 \right\rangle}}
\newcommand{\grad}{\ensuremath{\nabla}}

\let\emptyset\varnothing
\newcommand{\set}[1]{\left\{#1\right\}}

\newcommand{\bb}{\mathbb}

\def\N{{\mathbb{N}}}

\def\R{{\mathbb{R}}}
\def\S{{\mathbb{S}}}

\def\bS{{\mathbf{S}}}

\def\cB{{\cal B}}

\def\cF{{\cal F}}
\def\cG{{\cal G}}

\def\cM{{\cal M}}
\def\cN{{\cal N}}

\def\cR{{\cal R}}
\def\cS{{\cal S}}

\DeclareMathOperator{\Opt}{Opt}
\DeclareMathOperator*{\argmin}{arg\,min}
\DeclareMathOperator*{\argmax}{arg\,max}

\DeclareMathOperator{\Diag}{Diag}

\DeclareMathOperator{\tr}{tr}

\DeclareMathOperator*{\E}{\mathbb{E}}

\DeclareMathOperator{\spann}{span}
\DeclareMathOperator{\inter}{int}
\DeclareMathOperator{\rint}{rint}

\DeclareMathOperator{\bd}{bd}

\DeclareMathOperator{\conv}{conv}
\DeclareMathOperator{\cone}{cone}
\DeclareMathOperator{\clconv}{clconv}
\DeclareMathOperator{\clcone}{clcone}

 \usepackage{soul}

\newcommand{\faceeq}{\trianglelefteq}

\newcommand{\SDP}{\textup{SDP}}

\newcommand{\obj}{{\textup{obj}}}
\newcommand{\aas}{\textit{a.a.s.}}
\newcommand{\ngoe}{\textup{NGOE}}
\newcommand{\nsc}{\textup{nsc}}

\usepackage{tikz}

\usepackage{graphicx}
\usepackage[caption=false]{subfig}

\PassOptionsToPackage{numbers, compress}{natbib}

\newcommand{\mathprog}[1]{}

\begin{document}

\title{A geometric view of SDP exactness in QCQPs and its applications}
\author[1]{Alex L.\ Wang}
\author[1]{Fatma K\i l\i n\c{c}-Karzan}
\affil[1]{Carnegie Mellon University, Pittsburgh, PA, 15213, USA.}
\date{\today}

\maketitle

\begin{abstract}
Quadratically constrained quadratic programs (QCQPs) are a highly expressive class of nonconvex optimization problems. While QCQPs are NP-hard in general, they admit a natural convex relaxation via the standard (Shor) semidefinite program (SDP) relaxation. Towards understanding when this relaxation is exact, we study general QCQPs and their (projected) SDP relaxations. We present sufficient (and in some cases, also necessary) conditions for objective value exactness (the condition that the objective values of the QCQP and its SDP relaxation coincide) and convex hull exactness (the condition that the convex hull of the QCQP epigraph coincides with the epigraph of its SDP relaxation). Our conditions for exactness are based on geometric properties of $\Gamma$, the cone of convex Lagrange multipliers, and its relatives $\Gamma_P$ and $\Gamma^\circ$. These tools form the basis of our main message: questions of exactness can be treated systematically whenever $\Gamma$, $\Gamma_P$, or $\Gamma^\circ$ is well-understood. As further evidence of this message, we apply our tools to address questions of exactness for a prototypical QCQP involving a binary on-off constraint, quadratic matrix programs, the QCQP formulation of the partition problem, and random and semi-random QCQPs. \end{abstract}

\section{Introduction}
\label{sec:intro}

Quadratically constrained quadratic programs (QCQPs) are a fundamental class of \emph{nonconvex} optimization problems of the form
\begin{align*}
\Opt\coloneqq \inf_{x\in\R^n}\set{q_\obj(x):\, \begin{array}
	{l}
	q_i(x) \leq 0,\,\forall i\in[m_I]\\
	q_i(x) = 0,\,\forall i\in[m_I+1,m]
\end{array}},
\end{align*}
where $q_\obj, q_1,\dots,q_m:\R^n\to\R$ are each (possibly nonconvex) quadratic functions.
For each $i\in[m]$, we will write $q_i(x) = x^\top A_i x + 2b_i^\top x + c_i$ for $A_i\in\S^n$, $b_i\in\R^n$, and $c_i\in\R$. Similarly, write $q_\obj(x) = x^\top A_\obj x + 2b_\obj^\top x + c_\obj$.

These optimization problems arise naturally in a variety of application areas (see \cite{wang2020tightness,bao2011semidefinite,benTal2001lectures}).
Indeed, one fundamental reason for the ubiquity of QCQPs is their expressiveness---any polynomial optimization problem or $\set{0,1}$-integer program may be reformulated as a QCQP.

Although QCQPs are NP-hard in general, they admit a natural tractable convex relaxation known as the standard semidefinite program (SDP) relaxation \cite{shor1990dual},
\begin{align*}
\Opt_\SDP\coloneqq \inf_{x\in\R^n}\set{\ip{A_\obj, X} + 2b_\obj^\top x + c_\obj:\, \begin{array}
	{l}
	\exists X\succeq xx^\top\,:\\
	\ip{A_i,X} + 2b_i^\top x + c_i \leq0,\,\forall i\in[m_I]\\
	\ip{A_i, X} + 2b_i^\top x + c_i = 0 ,\,\forall i\in[m_I + 1,m]
\end{array}}.
\end{align*}
This relaxation is also referred to as the Shor SDP relaxation.
In contrast to the vast literature on the \emph{approximation} quality of this relaxation~\cite{benTal2001lectures,megretski2001relaxations,nesterov1997quality,ye1999approximating}, the question of when \emph{exactness} occurs in this relaxation is much more limited and recent.

One interesting line of work has offered deterministic conditions under which the SDP relaxation of a general QCQP is exact for various definitions of exactness.
In their celebrated paper, \citet{fradkov1979s-procedure} prove the S-lemma, which implies that the problem of minimizing an arbitrary quadratic objective function over the unit ball (or any single quadratic constraint)
can be solved via SDP techniques. Specifically, the S-lemma implies that \textit{objective value exactness}---the condition that the optimal value of the QCQP and the optimal value of its SDP relaxation coincide---holds for QCQPs with a single constraint; see also \cite{wang2020generalized}.
In contrast, \citet{burer2019exact} study diagonal QCQPs---those QCQPs for which $A_\obj,A_1,\dots,A_m$ are diagonal matrices---with a \emph{general} number of constraints and give sufficient conditions for objective value exactness.
\citet{wang2020tightness,wang2020convex} continue this line of work by developing a general framework for deriving sufficient conditions for both objective value exactness and \textit{convex hull exactness}---the condition that the convex hull of the QCQP epigraph coincides with the (projected) SDP epigraph---for QCQPs with a polyhedral set of projective convex Lagrange multipliers $\Gamma_P$ (see \cref{sec:prelim}).
Beyond being a natural sufficient condition for objective value exactness, convex hull exactness has its own far-reaching applications and motivation.
Such results find use for example in deriving strong relaxations of certain critical substructures in nonconvex problems. Specifically, the convexification of commonly occurring substructures in complex nonconvex problems has been critical in advancing the state-of-the-art computational approaches for general nonlinear nonconvex programs and mixed integer linear programs \cite{conforti2014integer,tawarmalani2002convexification}.
(See \cite{wang2020tightness,argue2020necessary} and references therein for additional work in this direction.)

While the framework presented by \citet{wang2020tightness} can at once cover and extend many existing results on objective value and convex hull exactness \cite{locatelli2015some,burer2019exact,fradkov1979s-procedure,hoNguyen2017second,wang2020generalized,modaresi2017convex,yildiran2009convex,burer2017how}, it is still quite limited. In particular, the assumption that the set of Lagrange multipliers $\Gamma_P$ is polyhedral is rarely satisfied outside of simultaneously diagonalizable QCQPs and precludes the results in \cite{wang2020tightness} from being applicable to a wider range of interesting QCQPs.

Additional work in this direction~\cite{cifuentes2020geometry} studies objective value exactness from an algebraic point of view. 
Specifically, \citet{cifuentes2020geometry} consider QCQPs with fixed equality constraints and study the semialgebraic region of objective functions for which objective value exactness holds. As an example of their results, they give a formula for the degree of the algebraic boundary of this region in the setting of Euclidean distance minimization problems.

A related line of work has explored sufficient conditions for the \emph{rank-one-generated} (ROG) property~\cite{argue2020necessary,burer2015gentle,burer2013second,blekherman2017sums,hildebrand2016spectrahedral}. Recall that a conic subset of the positive semidefinite cone is said to be ROG if it is the convex hull of its rank-one elements. 
This property can be thought of as the SDP--QCQP analogue to the integrality property in the context of linear program relaxations of integer linear programs~\cite{argue2020necessary} and can be shown to imply both convex hull exactness and objective value exactness.
Research in this direction has established explicit descriptions of the ROG cones related to quadratic programs over low-dimensional polytopes~\cite{burer2013second} and ellipsoids with missing caps~\cite{burer2014trust}.
Other work in this direction~\cite{hildebrand2016spectrahedral,blekherman2017sums} explores the ROG property from an algebro-geometric perspective and establishes results related to the degree and representation of such sets.
More recently, \citet{argue2020necessary} gave general sufficient conditions for this property and completely characterized the ROG cones defined by at most two linear matrix inequalities.

SDP exactness has been studied in the context of quadratic matrix programs (QMPs) as well. A QMP is an optimization problem over a matrix variable $X\in\R^{r\times k}$, where the objective function and constraints are each of the form
\begin{align*}
\tr(X^\top A X) + 2\tr(B^\top X) + c
\end{align*}
for $A\in\S^{r}$, $B\in\R^{r\by k}$ and $c\in\R$, and can be thought of as a natural generalization to QCQPs.
This class of problems has been used to model robust least squares problems, the orthogonal Procrustes problem \cite{beck2007quadratic}, and sphere packing \cite{beck2012new}. 
QMPs and their SDP relaxations were first studied by \citet{beck2007quadratic,beck2012new} who showed that objective value exactness holds as long as the number of constraints is small compared to $k$.
Similarly, \citet{wang2020tightness} show that both objective value exactness and convex hull exactness hold
for (vectorized reformulations of) QMPs whenever the number of constraints is small enough and $\Gamma_P$ is polyhedral.

Finally, a number of exciting results have shown that various \emph{random} QCQPs have exact SDP relaxations with high probability.
For example, such results have been proved for signal-recovery tasks such as phase retrieval~\cite{candes2015phase},
sensor-network localization~\cite{shamsi2013conditions},
max-likelihood angular synchronization~\cite{bandeira2017tightness},
and clustering~\cite{mixon2016clustering,abbe2015exact,rujeerapaiboon2019size}. In these settings, the goal is to recover some ground-truth solution (the solution to some QCQP) via observations (constraints in a QCQP). These results then show that once an application-specific signal-to-noise ratio is large enough (for example, given enough observations/constraints), that the SDP relaxation is exact.
In contrast, a second line of work~\cite{locatelli2020kkt,burer2019exact} addresses random QCQPs which do not assume the existence of a ground-truth solution. In this direction, it is shown that when the number of constraints is \emph{small enough} that the SDP relaxation has a rank-one optimal solution.

\subsection{Overview and outline of the paper}

In this paper, we vastly generalize the framework first introduced in \cite{wang2020tightness,wang2020convex} by eliminating its reliance on the polyhedrality assumption.
Specifically, we give a broad set of sufficient conditions for both convex hull exactness and objective value exactness that are phrased \emph{in terms of the set of projective convex Lagrange multipliers $\Gamma_P$} (or the closely related sets $\Gamma$ and $\Gamma^\circ$; see \cref{sec:prelim}). In particular, these sufficient conditions can be checked in a systematic manner whenever $\Gamma$, $\Gamma_P$, or $\Gamma^\circ$ is sufficiently simple.
Furthermore, we show that our sufficient conditions for convex hull exactness are additionally \emph{necessary} under a technical assumption (see \cref{as:facially_exposed}).
We complement our high-level theory with a number of explicit examples illustrating our tools on QCQPs from various settings, including a basic QCQP originating from modeling big-M constraints, quadratic matrix programs, the partition QCQP, and two random QCQP models. 

Collectively, these results and examples offer evidence for the main message of this paper that \emph{questions of exactness can be treated systematically whenever
the convex Lagrange multipliers are well-understood.}

A summary of our contributions, along with an outline of the remainder of the paper, is as follows:

\begin{enumerate}
	\item In \cref{sec:prelim}, we formally define our setup and assumptions and recall basics regarding Lagrangian aggregation and the SDP relaxation of a QCQP. We then define and examine a number of faces of the cone of convex Lagrange multipliers $\Gamma$ and its polar cone $\Gamma^\circ$ that play key roles in our analysis.
\item In \cref{sec:conv_hull_exactness}, we present a sufficient condition for convex hull exactness that generalizes \cite[Theorem 1]{wang2020tightness}.
	This sufficient condition (\cref{thm:conv_hull_sufficient}) is based on an analysis of the ``rounding directions'' inside $\cS_\SDP$ and is performed in the original space. 
	Specifically, we show that convex hull exactness holds as long as certain systems of equations (that depend on $\Gamma$, $\Gamma_P$, or $\Gamma^\circ$) contain nontrivial solutions.
	In contrast to \cite[Theorem 1]{wang2020tightness}, our sufficient condition does not make any assumptions on the geometry of $\Gamma$ and can be used to cover additional interesting QCQPs (see \cref{sec:conv_applications}).
	One of our main technical contributions (\cref{thm:conv_nec}) shows that our sufficient condition for convex hull exactness is in fact also \emph{necessary} under the assumption that $\Gamma^\circ$ is facially exposed (see \cref{as:facially_exposed} and its surrounding discussion). 
	We end \cref{sec:conv_hull_exactness} by revisiting the polyhedral setting. We derive necessary and sufficient conditions for convex hull exactness (\cref{thm:polyhedral}) and compare it to the sufficient condition presented in \cite[Theorem 1]{wang2020tightness}. To the best of our knowledge, this is the first necessary and sufficient condition for convex hull exactness even in the context of diagonal QCQPs (where $\Gamma,\Gamma_P$ and $\Gamma^\circ$ are automatically polyhedral).

\item In \cref{sec:conv_applications}, we present example applications of our general results from \cref{sec:conv_hull_exactness} to
a prototypical set containing big-M constraints, quadratic matrix programs, and the partition problem. In all of these applications, the resulting $\Gamma$ sets are non-polyhedral, and thus the sufficient conditions from \cite{wang2020tightness} that work under the polyhedrality assumption of  $\Gamma$ fail to cover these applications. 

In \cref{sec:mixed_binary_programming}, we apply our framework to show that convex hull exactness holds for a well-studied set involving convex quadratics, binary variables and big-M relations. This set occurs as a substructure commonly studied in sparse regression applications. The convex hull characterization of this set is well-known in the literature and is often shown as a consequence of the perspective formulation trick due to \citet{ceria1999convex} (see also \cite{frangioni2006perspective,gunluk2010perspective,dong2013valid}).  

In \cref{subsec:quadratic_matrix_programming}, we show that the SDP relaxation of a quadratic matrix program satisfies convex hull exactness whenever the number of constraints is small (when compared to the rank of the matrix variable). This strengthens separate results first presented in \cite{wang2020tightness} and \cite{beck2007quadratic}; see \cref{rem:qmp}. In contrast to the \textit{ad hoc} proof given in \cite{wang2020tightness}, the proof we present in \cref{subsec:quadratic_matrix_programming} follows the outline of our general framework.

In \cref{sec:partition_problem}, we consider the QCQP formulation of the NP-hard partition problem and its SDP relaxation. Using our framework, we give an explicit description of the optimal value  and epigraph of the SDP relaxation. Consequently, we recover a result due to \citet{laurent1995positive} stating that \emph{deciding} whether objective value exactness holds for the partition QCQP is NP-hard. In contrast,
	we show that convex hull exactness never holds for the partition QCQP (as long as there are at least two nonzero weights). This then implies that deciding whether convex hull exactness holds for the partition QCQP is trivial.

\item In \cref{sec:obj_val_exactness}, we present a number of sufficient conditions for objective value exactness. In fact, our sufficient conditions further imply \emph{optimizer exactness}, i.e., that the optimizers of the QCQP and its (projected) SDP relaxation coincide.
\cref{subsec:obj_primal} presents a general sufficient condition (\cref{thm:suff_obj_primal}) for objective value exactness based on a primal analysis.
Similarly, \cref{subsec:obj_dual} presents a general sufficient condition (\cref{lem:suff_obj_dual}) for objective value exactness based on a dual analysis.
These results recover known sufficient conditions~\cite{wang2020tightness,burer2019exact} for objective value exactness and explain the roles played by polyhedrality in prior settings.
We additionally specialize these abstract conditions to derive more concrete conditions (see \cref{cor:one_shot_obj,cor:polyhedral_suff_primal_obj,cor:suff_obj_dual,cor:sufficiently_steep}) for objective value exactness.

\item In \cref{sec:obj_appl}, we present example applications of our general results from \cref{sec:obj_val_exactness} to two models of random QCQPs.
The results in this section show that ideas from \cref{sec:obj_val_exactness} can be applied even when $\Gamma$, $\Gamma_P$, or $\Gamma^\circ$ is only known approximately.
The models in this section are inspired by recent work on objective value exactness~\cite{burer2019exact,locatelli2020kkt} where random QCQPs have been used as a testing ground for understanding the strength or explanatory power of various sufficient conditions.
In \cref{subsec:fully_gaussian}, we consider a fully random model of QCQPs and show that objective value exactness (in fact optimizer exactness) holds with probability $1-o(1)$ in the regime where $m$ (the number of constraints) is fixed and $n$ (the number of variables) diverges to $+\infty$.
In \cref{subsec:semi_random}, we consider a semi-random model of QCQPs where, for each quadratic function, the quadratic terms are randomly generated and the linear and constant terms can be chosen adversarially.
In this setting, we show that a perturbed notion of exactness holds again with probability $1-o(1)$ as $n\to+\infty$.

\end{enumerate}
 \subsection{Notation}
For $x,y\in\R$, let $[x\pm y]\coloneqq [x-y, x+y]$, $x_+\coloneqq \max(0,x)$ and $x_+^2 \coloneqq (x_+)^2$.
Let $n$ a positive integer.
Let $[n]\coloneqq\set{1,\dots,n}$ and for $i\in[n]$, let $e_i$ denote the $i$th standard basis vector in $\R^n$.
Let $0_n$ denote the zero vector in $\R^n$. For $\delta\geq 0$ and $x\in\R^n$, let $B_n(x,\delta) \coloneqq \set{y\in\R^n:\, \norm{x-y}\leq \delta}$. When $n$ is clear from context, we will simply write $0$ and $B(x,\delta)$.
Let $\R^n_+$ (resp.\ $\R^n_{++}$) denote the entrywise nonnegative (resp.\ positive) vectors in $\R^n$. Similarly define $\R^n_-$ and $\R^n_{--}$.
Let $\bS^{n-1}\subseteq\R^n$ denote the unit sphere.
Let $\S^n$ denote the vector space of $n\by n$ real symmetric matrices and $\S^n_+$ the cone of positive semidefinite matrices.
For $M\in\S^n$, we write $M\succeq 0$ (resp.\ $M\succ 0$) to denote that $M$ is positive semidefinite (resp.\ positive definite). Let $\lambda_{\min}(M)=\lambda_1(M)\leq \dots\leq\lambda_n(M) =\lambda_{\max}(M)$ denote the spectrum of $M$ and let $\ker(M)$ denote the kernel of $M$.
For $x\in\R^n$, let $\Diag(x)\in\S^n$ denote the diagonal matrix with $\Diag(x)_{i,i} = x_i$ for all $i\in[n]$.
Let $\mathbb{E}$ denote an arbitrary Euclidean space. Given $\cM\subseteq \mathbb{E}$, let $\inter(\cM)$, $\bd(\cM)$, $\conv(\cM)$, $\clconv(\cM)$, $\cone(\cM)$, $\clcone(\cM)$, $\spann(\cM)$, $\cM^\perp$, $\rint(\cM)$, and $\dim(\cM)$
denote the
interior, boundary, convex hull, closed convex hull, conic hull, closed conic hull, span (linear hull), orthogonal complement, relative interior, and dimension of $\cM$ respectively.
Let $K\subseteq\mathbb{E}$ be a cone.
Let $K^\circ$ denote the polar cone of $K$. The notation $F\faceeq K$ denotes that $F$ is a face of $K$. By convention, faces of cones are always nonempty.
$C_c^\infty(\R^n)$ denotes the smooth functions with compact support on $\R^n$.
Let $\grad$ denote the gradient operator.
Let $N(\mu,\Sigma)$ denote the multivariate normal distribution with mean $\mu$ and covariance $\Sigma$. 

\section{Preliminaries}
\label{sec:prelim}

\subsection{Setup}
\label{subsec:setup_notation}

We will consider quadratically constrained quadratic programs (QCQPs) in $\R^n$ defined by $m$-many quadratic constraints
\begin{align}
\label{eq:qcqp}
\Opt \coloneqq \inf_{x\in\R^n}\set{q_\obj(x):\, \begin{array}
	{l}
	q_i(x) \leq 0,\,\forall i\in[m_I]\\
	q_i(x) = 0 ,\,\forall i\in[m_I +1, m]
\end{array}}.
\end{align}
Here, $m_I$ is the number of inequality constraints and $m_E \coloneqq m - m_I$ is the number of equality constraints.
For each $i\in[m]$, we will write $q_i(x) = x^\top A_i x + 2b_i^\top x + c_i$ for some $A_i \in\S^n$, $b_i\in\R^n$, and $c_i\in\R$. Similarly, we will write $q_\obj(x) = x^\top A_\obj x + 2b_\obj^\top x  + c_\obj$.

We will also consider the epigraph, $\cS$, of this QCQP, i.e.,
\begin{align*}
\cS \coloneqq \set{(x,t)\in\R^n\times \R:\, \begin{array}
	{l}
	q_\obj(x) \leq 2t\\
	q_i(x) \leq 0,\,\forall i\in[m_I]\\
	q_i(x) = 0 ,\,\forall i\in[m_I + 1, m]
\end{array}}.
\end{align*}

\subsection{Aggregation and the (projected) SDP relaxation}
\label{subsec:convex_lagrange_multipliers_sdp}

It is well known in the QCQP literature~\cite{benTal2001lectures,wang2020tightness,fujie1997semidefinite} that the SDP relaxation of a QCQP is equivalent (under a minor assumption) to the double-Lagrangian-dual. We will state this formally in \cref{lem:S_SDP} but will first need to introduce notation related to Lagrangian aggregation. 

Let $q:\R^n\to\R^{1+m}$ be indexed by $\set{\obj}\cup[m]$ where $q(x)_\obj = q_\obj(x)$ and $q(x)_i = q_i(x)$ for $i\in[m]$.
Let $e_\obj, e_1,\dots,e_m$ denote the corresponding unit vectors in $\R^{1+m}$. 
We will work extensively with the aggregated quadratic functions $\ip{(\gamma_\obj,\gamma),q(x)}$ for $(\gamma_\obj,\gamma)\in\R^{1+m}$.
For notational convenience, define $A(\gamma_\obj,\gamma)\coloneqq \gamma_\obj A_\obj + \sum_{i\in[m]}\gamma_i A_i$. Similarly define $b(\gamma_\obj,\gamma)$, and $c(\gamma_\obj,\gamma)$.
We will at times work in the \emph{projective} version of the dual space where the distinguished variable $\gamma_\obj$ is taken to be one.
Let $A[\gamma]\coloneqq A(1,\gamma)$ and similarly define $b[\gamma]$ and $c[\gamma]$. Set $[\gamma,q(x)] \coloneqq \ip{(1,\gamma),q(x)}$.
Note that
\begin{align*}
	\ip{(\gamma_\obj,\gamma),q(x)} &=
\gamma_\obj q_\obj(x) + \sum_{i=1}^m \gamma_i q_i(x)\\
&=
x^\top A(\gamma_\obj,\gamma) x + 2b(\gamma_\obj,\gamma)^\top x + c(\gamma_\obj,\gamma),\quad\text{and}\\
[\gamma,q(x)] &=	q_\obj(x) + \sum_{i=1}^m \gamma_i q_i(x)\\
&=
x^\top A[\gamma] x + 2b[\gamma]^\top x + c[\gamma].
\end{align*}

We recall and extend the following definition from~\cite{wang2020tightness}.
\begin{definition}
\label{def:convex_lagrange_mulipliers}
The \textit{cone of convex Lagrange multipliers} for \eqref{eq:qcqp} is
\begin{align*}
\Gamma &\coloneqq \set{(\gamma_\obj,\gamma)\in\R\times\R^m:\, \begin{array}
	{l}
	A(\gamma_\obj,\gamma)\succeq 0\\
	\gamma_\obj \geq 0\\
	\gamma_i \geq 0 ,\,\forall i\in[m_I]
\end{array}}.
\end{align*}
The \textit{set of projective convex Lagrange multipliers} for \eqref{eq:qcqp} is
\begin{align*}
\Gamma_P &\coloneqq \set{\gamma\in\R^m:\, (1,\gamma)\in\Gamma} = \set{\gamma\in\R^m:\, \begin{array}
	{l}
	A[\gamma]\succeq 0\\
	\gamma_i \geq 0,\,\forall i\in[m_I]
\end{array}}.\qedhere
\end{align*}\mathprog{\qed}
\end{definition}
Note that given $(\gamma_\obj,\gamma)\in\Gamma$, the quadratic function $x\mapsto \ip{(\gamma_\obj,\gamma),q(x)}$ is convex. Similarly, given $\gamma\in\Gamma_P$, the quadratic function $x\mapsto [\gamma,q(x)]$ is convex.

We will make the following \emph{blanket assumption} for the remainder of the paper. This assumption can be interpreted as a dual strict feasibility condition and is  standard in the literature \cite{beck2007quadratic,benTal1996hidden,ye2003new,burer2019exact,wang2020tightness}.

\begin{assumption}
\label{as:definite}
There exists $(\gamma_\obj^*,\gamma^*)\in\Gamma$ such that $A(\gamma_\obj^*,\gamma^*)\succ 0$. Equivalently, there exists $\gamma^*\in\Gamma_P$ such that $A[\gamma^*]\succ 0$.
\end{assumption}

\begin{remark}
\label{rem:Gamma_generated_by_Gamma_P}
Note that under \cref{as:definite}, we have that $\Gamma$ is the closed cone generated by its slice at $\gamma_\obj=1$, i.e., $\Gamma = \clcone\left(\set{(1,\gamma):\, \gamma\in\Gamma_P}\right)$. (See discussion following \cite[Assumption 2]{wang2020tightness})\mathprog{\qed}
\end{remark}

Recall that the (projected) \emph{SDP relaxation} of $\cS$ is given by
\begin{align}
\label{eq:sdp_relaxation_primal}
\cS_\SDP \coloneqq \set{(x,t)\in\R^n\times\R:\, \begin{array}
	{l}
	\exists X\succeq xx^\top:\\
	\ip{A_\obj, X} + 2b_\obj^\top x + c_\obj \leq 2t\\
	\ip{A_i, X} + 2b_i^\top x + c_i \leq 0 ,\,\forall i\in[m_I]\\
	\ip{A_i, X} + 2b_i^\top x + c_i = 0 ,\,\forall i\in[m_I+1,m]\\\end{array}} ,
\end{align}
and $\Opt_\textup{SDP}\coloneqq \inf_{(x,t)\in\cS_\SDP}2t$.
By taking $X = xx^\top$ in \eqref{eq:sdp_relaxation_primal}, we see that
$\Opt\geq \Opt_\SDP$ and
$\conv(\cS) \subseteq \cS_\SDP$.

The following lemma states that under \cref{as:definite}, we can rewrite $\cS_\SDP$ in terms of $\Gamma$.
This lemma follows from a straightforward duality argument.
\begin{restatable}{lemma}{lemSSDP}
\label{lem:S_SDP}
Suppose \cref{as:definite} holds. Then
\begin{align*}
\cS_\SDP &= \set{(x,t)\in\R^{n+1}:\, [\gamma,q(x)] \leq 2t ,\,\forall \gamma\in\Gamma_P}\\
&= \set{(x,t)\in\R^{n+1}:\, \ip{(\gamma_\obj,\gamma), q(x)}\leq 2\gamma_\obj t,\,\forall (\gamma_\obj,\gamma)\in\Gamma}\\
& = \set{(x,t)\in\R^n:\, q(x) - 2te_\obj \in\Gamma^\circ}.
\end{align*}
Here, $\Gamma^\circ$ denotes the polar cone of $\Gamma$.
\end{restatable}
\begin{proof}
Fix $(x,t)\in\R^{n+1}$. Note that
\begin{align*}
\sup_{\gamma\in\Gamma_P} [\gamma,q(x)] &= \sup_{\gamma\in\R^m}\set{[\gamma, q(x)] :\,	\begin{array}
	{l}
	A[\gamma]\succeq 0\\
	\gamma_i \geq 0 ,\,\forall i\in[m_I]
\end{array}}\\
&= \inf_{\xi\in\S^n} \set{q_\obj(x) + \ip{A_\obj, \xi}:\, \begin{array}
	{l}
	q_i(x) + \ip{A_i, \xi} \leq 0 ,\,\forall i\in[m_I]\\
	q_i(x) + \ip{A_i, \xi} = 0 ,\,\forall i\in[m_I+1,m]\\
	\xi \succeq 0
\end{array}},
\end{align*}
where the second equation follows from the strong conic duality theorem and \cref{as:definite}. Taking $X \coloneqq xx^\top + \xi$, we deduce that the first equality in \cref{lem:S_SDP} holds.

Note that by \cref{as:definite}, $\Gamma = \clcone\set{(1,\gamma):\, \gamma\in\Gamma_P)}$ so that $[\gamma,q(x)]\leq 2t$ for all $\gamma\in\Gamma_P$ if and only if $\ip{(\gamma_\obj,\gamma),q(x)}\leq 2\gamma_\obj t$ for all $(\gamma_\obj,\gamma)\in\Gamma$; this gives the second equality. The third equality holds by definition of the polar cone.
\end{proof}

\begin{corollary}
\label{cor:opt_sdp_saddle}
Suppose \cref{as:definite} holds. Then
\begin{align}\label{eq:SDP_SP}
\Opt_\SDP = \inf_{x\in\R^n}\sup_{\gamma\in\Gamma_P} [\gamma,q(x)].
\end{align}
\end{corollary}

\begin{corollary}
\label{cor:SSDP_closed}
Suppose \cref{as:definite} holds. Then, $\cS_\SDP$ is closed.
\end{corollary}

\begin{remark}\label{rem:explicit_gamma_polar_in_overline_cS}
In comparison with \eqref{eq:sdp_relaxation_primal}, the expressions for $\cS_\SDP$ given in \cref{lem:S_SDP} make the roles played by $\Gamma$, $\Gamma_P$, and $\Gamma^\circ$ explicit. In particular, these expressions for $\cS_\SDP$ lend themselves to a clean analysis whenever the corresponding dual set $\Gamma$, $\Gamma_P$, or $\Gamma^\circ$ is sufficiently simple.\mathprog{\qed}
\end{remark}

\begin{remark}
Phrased differently, one may minimize $\Opt_\SDP$ in the form \eqref{eq:sdp_relaxation_primal} by minimizing over $x\in\R^n$ the value of an inner minimization problem over the matrix variables $X\succeq xx^\top\in\S^n$. Writing $X = xx^\top + \xi$ and taking the SDP dual in the $\xi$ variable then results in the same saddle-point structure $\Opt_\SDP = \inf_{x\in\R^n}\sup_{\gamma\in\Gamma_P} [\gamma,q(x)]$ observed in \cref{cor:opt_sdp_saddle}. In other words, $\Gamma_P$ is simply the set of feasible solutions to this partial dual of \eqref{eq:sdp_relaxation_primal}.\mathprog{\qed}
\end{remark}

Let us consider a concrete example to help materialize these definitions.

\begin{example}
\label{ex:explicit_overline_cS}
Consider the following QCQP epigraph,
\begin{align*}
\cS\coloneqq \set{(x,t)\in\R^2\times\R:\, \begin{array}
	{l}
	q_\obj(x)\leq 2t\\
	q_1(x)\leq 0\\
	q_2(x)\leq 0
\end{array}},
\end{align*}
where $q_\obj(x) \coloneqq 2x_1x_2 - x_2 - 1/4$,~ $q_1(x) \coloneqq x_1^2 - x_2^2 - x_1 + x_2 - 1$, and $q_2(x)\coloneqq x_1^2 + x_2^2 -1$. Through a straightforward calculation, we obtain 
\begin{gather*}
\Gamma = \set{(\gamma_\obj,\gamma)\in\R^3:\, \begin{array}
	{l}
	\gamma_2 \geq \sqrt{\gamma_\obj^2 + \gamma_1^2}\\
	\gamma_\obj,\gamma_1,\gamma_2\geq 0
\end{array}},\\
\Gamma^\circ = \set{(\gamma_\obj,\gamma)\in\R^3:\, -\gamma_2 \geq \sqrt{(\gamma_\obj)_+^2 + (\gamma_1)_+^2}}, \text{ and}\\
\cS_\SDP = \set{(x,t)\in\R^2:\, -q_2(x) \geq \sqrt{(q_\obj(x) - 2t)_+^2 + q_1(x)_+^2}}.
\end{gather*}
See \cref{fig:example_overline_cS} for the plots of the sets corresponding to $\cS$, $\Gamma$, $\Gamma^\circ$, and $\cS_\SDP$.\mathprog{\qed}
\end{example}
\begin{figure}
\centering
		\begin{tikzpicture}
			\tikzset{edge/.style = {->,> = latex}}

			\node at (0,0) {\includegraphics[width=87.5pt]{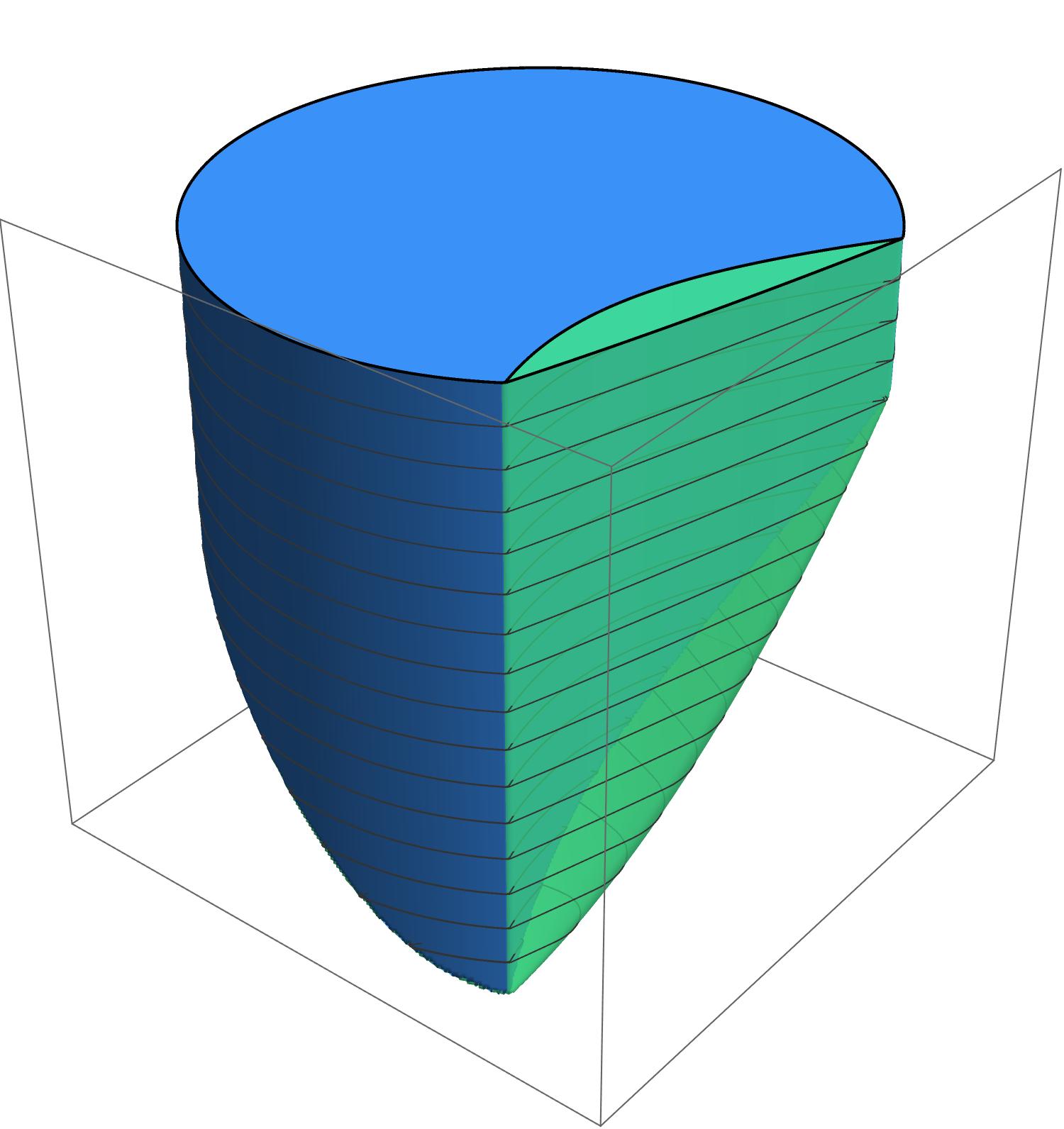}};
			\node at (5,0) {\includegraphics[width=87.5pt]{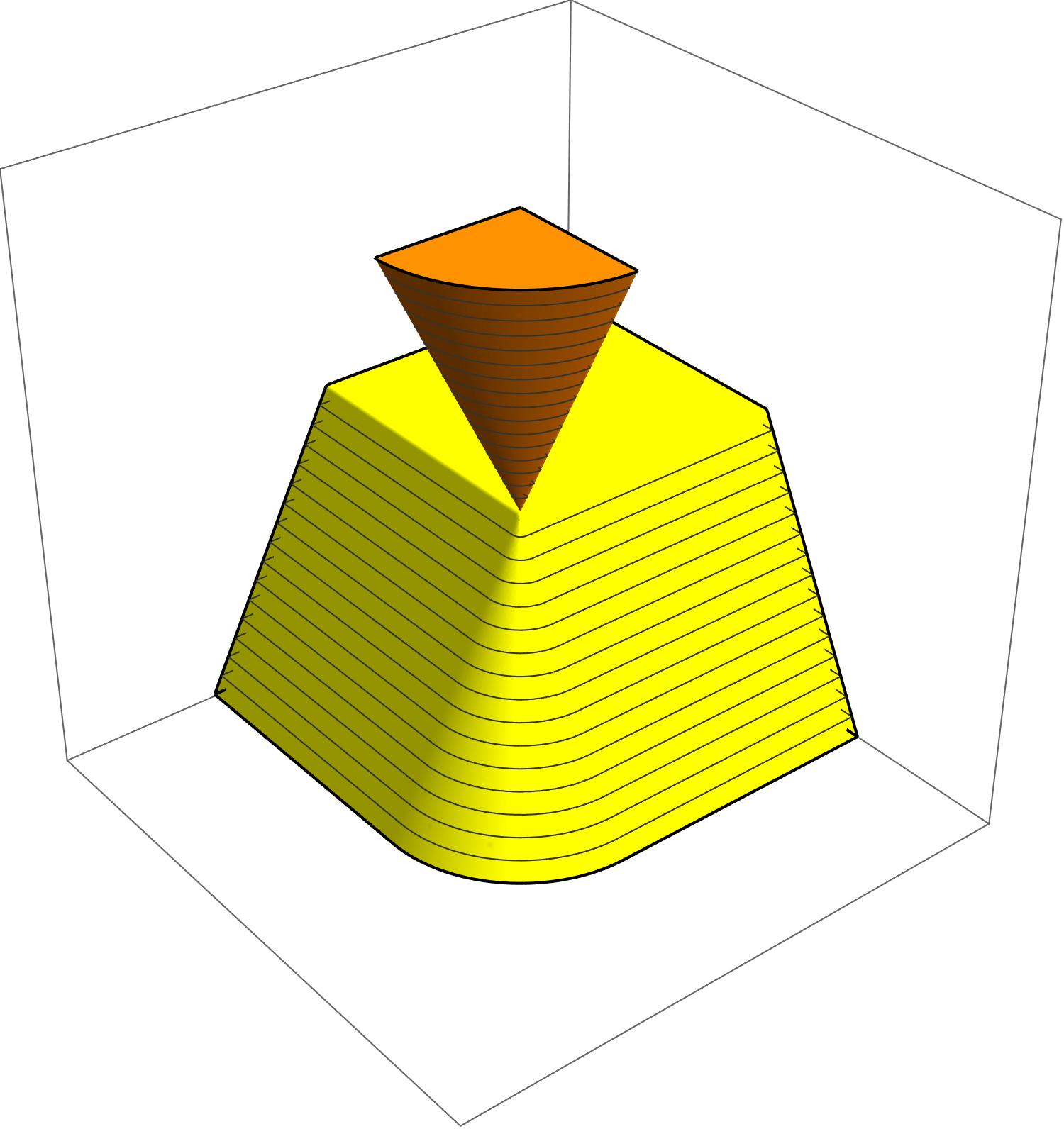}};

			\node (a) at (0.75,0.25) [circle,fill,inner sep=1pt]{};
			\node (b) at (5.75,-0.25) [circle,fill,inner sep=1pt]{};
			
			\path[->, > = latex](a) edge[bend left = 20] node[above] {} (b);

			\node at (-1,1.5) {$\cS$};
			\node at (1.5,0.75) {$\cS_\SDP$};
			\node at (5.5,1) {$\Gamma$};
			\node at (5.5,-1) {$\Gamma^\circ$};
			
			\node at (1.25,0) {$(x,t)$};
			\node at  (7.25,-0.25) {$(q_\obj(x)-t,q(x))$};
		\end{tikzpicture}
\caption{The sets $\cS$, $\cS_\SDP$, $\Gamma$, and $\Gamma^\circ$ from \cref{ex:explicit_overline_cS} are shown in blue, green, orange, and yellow respectively. By \cref{lem:S_SDP}, $(x,t)\in\cS_\SDP$ if and only if $q(x) - 2te_\obj \in \Gamma^\circ$.}
\label{fig:example_overline_cS}
\end{figure}
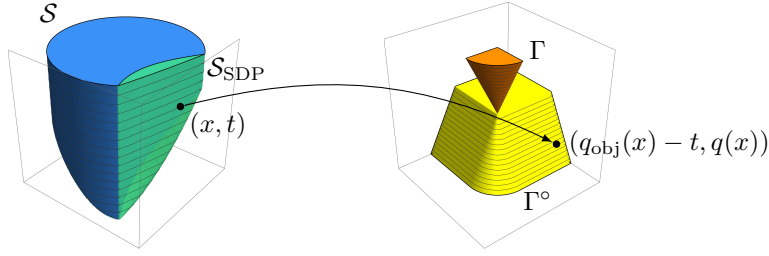

\subsection{Faces of $\Gamma$ and $\Gamma^\circ$}
\label{subsec:faces_gamma_gamma_polar}

In this section we define key faces of $\Gamma$ and $\Gamma^\circ$ that will play important roles in our analysis. We will additionally recall a number of elementary properties of convex cones and their faces specialized to our setting. See \cite{barker1978faces,barker1981theory,pataki2000geometry} for a more in-depth treatment of general convex cones and their faces.

Recall the following definitions.
\begin{definition}
\label{def:conjugate_face}
Given a face $\cG\faceeq\Gamma^\circ$ and $(g_\obj,g)\in\rint(\cG)$, the \emph{conjugate face of $\cG$} is 
\[\cG^\triangle \coloneqq \Gamma \cap \cG^\perp = \Gamma\cap (g_\obj,g)^\perp.\] 
Similarly, define the \emph{conjugate face of $\cF$} for a face $\cF\faceeq \Gamma$.\mathprog{\qed}
\end{definition}
\begin{definition}
For a face $\cG\faceeq\Gamma^\circ$, we say that $\cG$ is \textit{exposed} if there exists $(\gamma_\obj,\gamma)\in\Gamma$ such that $\cG = \Gamma^\circ \cap (\gamma_\obj,\gamma)^\perp$.\mathprog{\qed}
\end{definition}

We will additionally associate faces of $\Gamma$ and $\Gamma^\circ$ to points $(x,t)\in\cS_\SDP$ as follows.
\begin{definition}
\label{def:face_G_F}
Given $(x,t)\in\cS_\SDP$, let $\cG(x,t)\faceeq\Gamma^\circ$ denote the minimal face of $\Gamma^\circ$ containing
$q(x) - 2te_\obj$ and define $\cF(x,t)\coloneqq \cG(x,t)^\triangle$.\mathprog{\qed}
\end{definition}

The next fact follows from \cref{def:face_G_F}.
\begin{fact}
\label{fact:q_in_rint_G}
Given $(x,t)\in\cS_\SDP$, we have that $q(x) - 2te_\obj \in\rint(\cG(x,t))$ and $\cF(x,t) = \Gamma \cap (q(x) - 2te_\obj)^\perp$.
\end{fact}

\section{Convex hull exactness}
\label{sec:conv_hull_exactness}

In this section, we present necessary and sufficient conditions for convex hull exactness, i.e., the property that $\conv(\cS) = \cS_\SDP$.
These results form the basis of our assertion that \emph{exactness can be treated systematically whenever $\Gamma$, $\Gamma_P$, or $\Gamma^\circ$ is well-understood.}

We begin by rephrasing convex hull exactness as a question regarding the existence of certain ``rounding directions.'' The following result follows from basic convex analysis.
\begin{lemma}
\label{lem:conv_iff_R_nontrivial}
Suppose \cref{as:definite} holds. Then, $\conv(\cS) = \cS_\SDP$ if and only if for every $(x,t)\in\cS_\SDP\setminus\cS$, there exists a nonzero $(x',t')\in\R^{n+1}$ and $\alpha>0$ such that
\begin{align*}
[(x,t)\pm\alpha (x',t')] \subseteq \cS_\SDP.
\end{align*}
\end{lemma}
\begin{proof}
Note that $\cS_\SDP$ is a closed convex set containing no lines. Also, one can easily check that $(0_n,1)$ is indeed a recessive direction of $\cS_\SDP$. Furthermore, $(0_n,1)$ is the only recessive direction of $\cS_\SDP$. To see this, let $\gamma^*$ be such that $A[\gamma^*]\succ 0$ (which exists by \cref{as:definite}) and consider any $(x',t')$ where $x'$ is nonzero. Then, for any $(\tilde x,\tilde t)\in\cS_\SDP$ and all $\alpha>0$ large enough, $2(\tilde t+\alpha t') < [\gamma^*, q(\tilde x+\alpha x')]$. Therefore, we deduce by \cite[Theorem 18.5]{rockafellar1970convex}, that $\cS_\SDP$ is the sum of the convex hull of its extreme points and the direction $(0_n,1)$. In particular $\conv(\cS) = \cS_\SDP$ if and only if $(x,t)$ is not extreme for every $(x,t)\in\cS_\SDP\setminus\cS$. By definition, $(x,t)$ is not extreme if and only if there exists $(x',t')$ and $\alpha>0$ such that $[(x,t)\pm \alpha (x',t')]\subseteq \cS_\SDP$.
\end{proof}

We capture the relevant set in \cref{lem:conv_iff_R_nontrivial} in the following definition.
\begin{definition}
The subspace of \emph{rounding directions} at $(x,t)\in\cS_\SDP$ is
\begin{align*}
\cR(x,t) &\coloneqq \set{(x',t')\in\R^{n+1}:\,
\exists \alpha>0\text{ s.t. } [(x,t)\pm \alpha (x',t')]\subseteq\cS_\SDP}.
\end{align*}
This set is \emph{nontrivial} if it contains a nonzero element.\mathprog{\qed}
\end{definition}

Note that $\cR(x,t)$ is in fact a subspace so that its name is justified. Indeed, $\cR(x,t)$ is a convex cone as $\cS_\SDP$ is convex. Furthermore, it holds that $-\cR(x,t) = \cR(x,t)$.

\begin{remark}
One may compare our \emph{rounding directions} to other similar definitions from elementary convex analysis~\cite[Section 5.1]{hiriart2004fundamentals}. Fix a point $(x,t)\in\cS_\SDP$ and $(x',t')\in\R^{n+1}$. Recall that $(x',t')$ is a \emph{feasible direction} if there exists $\alpha>0$ such that $[(x,t), (x,t) + \alpha (x',t')] \subseteq\cS_\SDP$. In particular, feasible directions are a \emph{unidirectional} notion, whereas rounding directions are \emph{bidirectional}.
Next, recall that $(x',t')$ is a \emph{tangent direction} if it is a limit of feasible directions. Again, tangent directions are unidirectional.\mathprog{\qed}
\end{remark}

\begin{remark}\label{rem:suffices_check_epigraph_tight}
Suppose $(x,t)\in\cS_\SDP$ and $2t > \sup_{\gamma\in\Gamma_P}[\gamma,q(x)]$. Then, by \cref{lem:S_SDP} there exists $\alpha>0$ such that $[(x,t)\pm \alpha(0_n,1)]\subseteq\cS_\SDP$. In particular, it suffices to verify the condition of \cref{lem:conv_iff_R_nontrivial} for points $(x,t)\in\cS_\SDP\setminus\cS$ for which $2t = \sup_{\gamma\in\Gamma_P}[\gamma,q(x)]$.\mathprog{\qed}
\end{remark}

\subsection{Sufficient conditions for convex hull exactness}
\label{subsec:conv_suff}
In this section we identify a particular subset of the rounding directions at $(x,t)\in\cS_\SDP$. This then leads to a sufficient condition for convex hull exactness, i.e., the condition that $\conv(\cS) = \cS_\SDP$.

\begin{definition}
Given $(x,t)\in\cS_\SDP$, define
\begin{align*}
\cR'(x,t) &\coloneqq \set{(x',t')\in\R^{n+1}:~ q(x+\alpha x') - 2(t+\alpha t') e_\obj \in\spann(\cG(x,t)),\,\forall \alpha\in\R}.\qedhere\mathprog{\qed}
\end{align*}
\end{definition}

\begin{lemma}
\label{lem:R'_subset}
Suppose \cref{as:definite} holds and $(x,t)\in\cS_\SDP$. Then, $\cR'(x,t)\subseteq\cR(x,t)$.
\end{lemma}
\begin{proof}
Let $(x,t)\in\cS_\SDP$ and $(x',t')\in\cR'(x,t)$. Then, by continuity and the fact that $q(x) - 2te_\obj \in \rint(\cG(x,t))$, there exists $\alpha>0$ such that
\begin{align*}
q(x+ \epsilon x') - 2(t + \epsilon t') e_\obj \in\cG(x,t)\subseteq \Gamma^\circ
\end{align*}
for all $\epsilon\in[\pm\alpha]$. By the third characterization of $\cS_\SDP$ in  \cref{lem:S_SDP}, we have that $[(x,t)\pm\alpha (x',t')]\subseteq \cS_\SDP$.
\end{proof}

\cref{lem:conv_iff_R_nontrivial,lem:R'_subset} immediately imply the following sufficient condition for convex hull exactness.
\begin{theorem}
\label{thm:conv_hull_sufficient}
Suppose \cref{as:definite} holds and that for all $(x,t)\in\cS_\SDP\setminus\cS$, the set $\cR'(x,t)$ is nontrivial. Then, $\conv(\cS) = \cS_\SDP$.
\end{theorem}

We will see a number of applications of \cref{thm:conv_hull_sufficient} in \cref{sec:conv_applications}.

In \cref{lem:R'_if_epigraph_tight} below, we will record an alternate description of $\cR'(x,t)$.
We will require the following observation.

\begin{observation}
\label{obs:epi_tight_clcone}
Suppose \cref{as:definite} holds. Let $(x,t)\in\cS_\SDP$ where $2t = \sup_{\gamma\in\Gamma_P} [\gamma,q(x)]$. Then, $\spann(\cG(x,t))\not\supseteq \R\times 0_m$.
In particular, $\cG(x,t)^\perp = \spann\left(\cG(x,t)^\perp \cap \set{\gamma_\obj = 1}\right)$.
\end{observation}
\begin{proof}
Suppose $\spann(\cG(x,t))\supseteq \R\times 0_m$ so that $(0_n, 1) \in\cR'(x,t)$. By \cref{lem:R'_subset}, there exists $\alpha>0$ such that  $(x,t - \alpha)\in\cS_\SDP$. This contradicts $2t = \sup_{\gamma\in\Gamma_p} [\gamma,q(x)]$. 

We deduce that $\spann(\cG(x,t))\not\supseteq\R\times0_m$. Equivalently, $\cG(x,t)^\perp \not\subseteq 0\times \R^m$ and there exists $(1,\bar\gamma)\in\cG(x,t)^\perp$.
Then, for any $(\gamma_\obj,\gamma)\in\cG(x,t)^\perp$, we can write $(\gamma_\obj,\gamma)$ as a linear combination of
\begin{align*}
&(\gamma_\obj,\gamma) + (1-\gamma_\obj) (1,\gamma)\qquad\text{and}
\qquad
(1,\gamma).\qedhere
\end{align*}
\end{proof}

\begin{lemma}
\label{lem:R'_if_epigraph_tight}
Suppose \cref{as:definite} holds and let $(x,t)\in\cS_\SDP$. Then,
\begin{align*}
\cR'(x,t) &= \set{(x',t')\in\R^{n+1}:\, \begin{array}
	{l}
	(x')^\top A(\gamma_\obj,\gamma) x' = 0  ,\,\forall (\gamma_\obj,\gamma)\in\cG(x,t)^\perp\\
	\ip{A(\gamma_\obj,\gamma)x + b(\gamma_\obj,\gamma), x'} - \gamma_\obj t' = 0,\,\forall (\gamma_\obj,\gamma)\in\cG(x,t)^\perp
\end{array}}
\end{align*}
If furthermore $2t = \sup_{\gamma\in\Gamma_P}[\gamma,q(x)]$, then
\begin{align*}
\cR'(x,t) &= \set{(x',t')\in\R^{n+1}:\, \begin{array}
	{l}
	(x')^\top A[\gamma] x' = 0  ,\,\forall (1,\gamma)\in\cG(x,t)^\perp\\
	\ip{A[\gamma]x + b[\gamma], x'} - t' = 0,\,\forall (1,\gamma)\in\cG(x,t)^\perp
\end{array}}.
\end{align*}
\end{lemma}
\begin{proof}
Note that $(x',t')\in\cR'(x,t)$ if and only if for all $(\gamma_\obj,\gamma)\in\cG(x,t)^\perp$, we have that
\begin{align*}
&\ip{(\gamma_\obj,\gamma), q(x + \alpha x') - 2(t+\alpha t')e_\obj}\\
&\quad = 
\alpha^2 (x')^\top A(\gamma_\obj,\gamma) (x') + 2\alpha\left(\ip{A(\gamma_\obj,\gamma) x + b(\gamma_\obj,\gamma), x'}- t'\right)\\
&\qquad  + \ip{(\gamma_\obj,\gamma), q(x) - 2te_\obj}
\end{align*}
is identically zero in $\alpha$. This occurs if and only if for all $(\gamma_\obj,\gamma)\in\cG(x,t)^\perp$, we have
\begin{gather*}
(x')^\top A(\gamma_\obj,\gamma) x' = 0,\quad\text{and}\\
\ip{A(\gamma_\obj,\gamma) x + b(\gamma_\obj,\gamma), x'} - t' = 0.
\end{gather*}
This proves the first assertion. The second assertion follows from the first and \cref{obs:epi_tight_clcone}.
\end{proof}

\subsection{Necessary conditions for convex hull exactness}
In \cref{subsec:conv_suff}, we gave a sufficient condition for convex hull exactness by identifying a subset of directions $\cR'(x,t)\subseteq\cR(x,t)$ and invoking \cref{lem:conv_iff_R_nontrivial}.
In this section, we show that under a technical assumption (\cref{as:facially_exposed}), we have $\cR'(x,t)=\cR(x,t)$. This then leads to a necessary and sufficient condition for convex hull exactness under the technical assumption.

\begin{assumption}
\label{as:facially_exposed}
Suppose $\Gamma^\circ$ is facially exposed, i.e., every face of $\Gamma^\circ$ is exposed.
\end{assumption}

This assumption holds for any cone isomorphic to a slice of the nonnegative orthant, the second-order cone, or the positive semidefinite cone. See \cite{pataki2013connection} for a longer discussion of this assumption and its connections to the \textit{nice cones}. In general, all \textit{nice cones} are facially exposed.
Our analysis will be based on the following property of exposed faces $\cG\faceeq\Gamma^\circ$ (see~\cite[Definition 2.A.9]{barker1981theory} and its surrounding discussion):
\begin{fact}
\label{fact:exposed_double_conjugate}
A face $\cG\faceeq\Gamma^\circ$ is exposed if and only if $\cG = (\cG^\triangle)^\triangle$.
\end{fact}

We are now ready to prove a partial converse to \cref{lem:R'_subset}.
\begin{lemma}
\label{lem:R'_equal}
Suppose \cref{as:definite,as:facially_exposed} hold and let $(x,t)\in\cS_\SDP$. Then $\cR'(x,t) = \cR(x,t)$.
\end{lemma}
\begin{proof}
Fix $(x',t')\in\cR(x,t)$. As $\cR(x,t)$ is a convex cone, we may without loss of generality assume that $[(x,t)\pm (x',t')]\subseteq \cS_\SDP$. Our goal is to show that $(x',t')\in\cR'(x,t)$, i.e., that
\begin{align*}
q(x+\alpha x') - 2(t+\alpha t')e_\obj \in\spann(\cG(x,t)),\,\forall \alpha\in\R.
\end{align*}
As each coordinate of this vector is quadratic in $\alpha$, it suffices to show instead that
\begin{align*}
q(x+\alpha x') - 2(t+\alpha t')e_\obj \in \cG(x,t),\,\forall \alpha\in[-1,1].
\end{align*}

Let $(f_\obj, f)\in\rint(\cF(x,t))$ so that by \cref{as:facially_exposed,fact:exposed_double_conjugate}, we may write $\cG(x,t) = \Gamma^\circ \cap (f_\obj, f)^\perp$.
As $[(x,t)\pm (x',t')]\subseteq \cS_\SDP$, we immediately have that $q(x+\alpha x') - 2(t+\alpha t')e_\obj\in \Gamma^\circ$ for all $\alpha \in[-1,1]$. It remains to verify that the map
\begin{align*}
\alpha\mapsto
\ip{
(f_\obj,f),
q(x+\alpha x') - 2(t+\alpha t')e_\obj}
\end{align*}
evaluates to zero on $\alpha \in[-1,1]$.
Again, as $[(x,t)\pm (x',t')]\subseteq \cS_\SDP$, this map is nonpositive for all $\alpha\in[-1,1]$.
Next, note that $(f_\obj,f)\in\cF(x,t) = \Gamma \cap (q(x) - 2te_\obj)^\perp$ so that this map evaluates to zero at $\alpha = 0$. Finally, $(f_\obj,f) \in \Gamma$ implies that this map is also convex. We conclude that this map is identically zero.
\end{proof}

The following necessary and sufficient condition for convex hull exactness then follows from \cref{lem:R'_equal}.
\begin{theorem}
\label{thm:conv_nec}
Suppose \cref{as:definite,as:facially_exposed} hold. Then, $\conv(\cS) = \cS_\SDP$ if and only if for all $(x,t)\in\cS_\SDP\setminus\cS$, the set $\cR'(x,t)$ is nontrivial.
\end{theorem}

To close this subsection, we give a compact description of $\cR'(x,t)$ under \cref{as:facially_exposed}.

\begin{proposition}
\label{prop:R'_exposed}
	Suppose \cref{as:definite,as:facially_exposed} hold. Let $(x,t)\in\cS_\SDP$ where $2t = \sup_{\gamma\in\Gamma_P}[\gamma,q(x)]$ and let $(1,f)\in\rint(\cF(x,t))$. Then,
	\begin{align*}
	\cR'(x,t) &= \set{(x',t')\in\R^{n+1}:\, \begin{array}
		{l}
		x' \in\ker(A[f])\\
		\ip{A[\eta]x + b[\eta],x'} - t' = 0 ,\,\forall (1,\eta)\in\cG(x,t)^\perp
	\end{array}}.
	\end{align*}
\end{proposition}
\begin{proof}
Let $(1,f)\in\rint(\cF(x,t))$.
By \cref{lem:R'_if_epigraph_tight}, it suffices to show that $x'\in\ker(A[f])$ if and only if
\begin{align*}
(x')^\top A[\gamma] x' = 0,\,\forall (1,\gamma)\in\cG(x,t)^\perp.
\end{align*}
The reverse direction holds immediately as $(1,f) \in \cF(x,t) \subseteq \cG(x,t)^\perp$ and $A[f]\succeq 0$.

To see the forward direction: Let $x'\in \ker(A[f])$ and set $v_\obj = (x')^\top A_\obj x'$. Similarly, set $v_i = (x')^\top A_i x'$. Then,
\begin{align*}
\ip{(v_\obj,v),(\gamma_\obj,\gamma)} &= (x')^\top A(\gamma_\obj,\gamma) x' \geq 0,\,\forall (\gamma_\obj,\gamma)\in\Gamma.
\end{align*}
Thus $(-v_\obj,-v)\in \Gamma^\circ$. On the other hand, $\ip{(v_\obj,v),(1,f)} = (x')^\top A[f] x' = 0$. We deduce that $(-v_\obj,-v) \in \Gamma^\circ \cap (1,f)^\perp = \cF(x,t)^\triangle = \cG(x,t)$. In particular, $(x')^\top A(\gamma_\obj,\gamma) x' = \ip{(v_\obj, v), (\gamma_\obj,\gamma)}= 0$ for all $(\gamma_\obj,\gamma)\in\cG(x,t)^\perp$.
\end{proof}

\subsection{Revisiting the setting of polyhedral $\Gamma$}
\label{subsec:polyhedral}

\citet{wang2020tightness} give sufficient conditions for convex hull exactness under the assumption that $\Gamma$ is polyhedral.
This assumption holds, for example, when the set of quadratic forms $\set{A_\obj,A_1,\dots,A_m}$ is simultaneously diagonalizable.
Specializing \cref{thm:conv_nec} to this setting, we 
prove the following \emph{necessary and sufficient} counterpart to \cite[Theorem 1]{wang2020tightness}.

\begin{restatable}
	{theorem}{thmpolyhedral}
	\label{thm:polyhedral}
	Suppose \cref{as:definite} holds and that $\Gamma$ is polyhedral.
	Then, $\conv(\cS) = \cS_\SDP$ if and only if
	\begin{align*}
	\set{(x',t')\in\R^{n+1}:\, \begin{array}
		{l}
		x' \in \ker(A[f])\\
		\ip{b[\gamma], x'} - t' = 0,\,\forall (1,\gamma) \in \cF
	\end{array}}
	\end{align*}
	is nontrivial for every $\cF\faceeq\Gamma$ which is exposed by some vector $q(x) - 2te_\obj$ for $(x,t)\in\cS_\SDP\setminus\cS$. Here, $f$ is any vector such that $(1,f)\in \rint(\cF)$.
\end{restatable}
\begin{proof}
We begin by noting that when $\Gamma$ is polyhedral, so too is $\Gamma^\circ$ so that \cref{as:facially_exposed} holds.
Next, we claim that for every face $\cG\faceeq\Gamma^\circ$
we have $\cG^\perp = \spann(\cG^\triangle)$. By definition, $\spann(\cG^\triangle) = \spann(\Gamma\cap \cG^\perp) \subseteq\cG^\perp$.
On the other hand, as $\Gamma$ and $\Gamma^\circ$ are polyhedral, we have that~\cite[Theorem 3]{tam1976note}
\begin{align*}
\dim(\cG) + \dim(\cG^\triangle) = m.
\end{align*}
Rearranging this equation, we have $\dim(\cG^\triangle) = m - \dim(\cG) = \dim(\cG^\perp)$. We conclude that $\cG^\perp = \spann(\cG^\triangle)$.

Let $(x,t)\in\cS_\SDP$ such that $2t = \sup_{\gamma\in\Gamma_P}[\gamma,q(x)]$ and let $(1,f)\in\rint(\cF(x,t))$. Then, 
\cref{obs:epi_tight_clcone} and \cref{prop:R'_exposed} imply that
\begin{align*}
\cR'(x,t) &= \set{(x',t')\in\R^{n+1}:\, \begin{array}
	{l}
	x'\in \ker(A[f])\\
	\ip{A[\gamma]x + b[\gamma],x'} - t' = 0 ,\,\forall (1,\gamma)\in\cG(x,t)^\perp
\end{array}}\\
&= \set{(x',t')\in\R^{n+1}:\, \begin{array}
	{l}
	x'\in \ker(A[f])\\
	\ip{A[\gamma]x + b[\gamma],x'} - t' = 0 ,\,\forall (1,\gamma)\in\cF(x,t)
\end{array}}\\
&= \set{(x',t')\in\R^{n+1}:\, \begin{array}
	{l}
	x'\in \ker(A[f])\\
	\ip{b[\gamma],x'} - t' = 0 ,\,\forall (1,\gamma)\in\cF(x,t)
\end{array}}.
\end{align*}
Here, the second line follows because we have shown $\cG^\perp = \spann(\cG^\triangle)$ holds for every face $\cG\faceeq\Gamma^\circ$ and by definition $\cF(x,t)=\cG(x,t)^\triangle$. The third line follows from the fact that $(1,f)\in\rint(\cF(x,t))$ implies $\ker(A[f])\subseteq\ker(A[\gamma])$ for every $(1,\gamma)\in\cF(x,t)$. 
The result then follows from \cref{thm:conv_nec}.
\end{proof}

\begin{remark}
The main difference between \cref{thm:polyhedral} and \cite[Theorem 1]{wang2020tightness} is that \cref{thm:polyhedral} only considers certain (\textit{a fortiori} semidefinite) faces of $\Gamma$ whereas \cite[Theorem 1]{wang2020tightness} imposes a constraint on \emph{every} semidefinite face of $\Gamma$. This idea of restricting the analysis to certain faces of $\Gamma$ was previously investigated by \cite{burer2019exact,locatelli2020kkt} who used it to provide sufficient conditions for objective value exactness.\mathprog{\qed}
\end{remark}  \section{Applications: Convex hull exactness}
\label{sec:conv_applications}

In this section, we apply the results of \cref{sec:conv_hull_exactness} to a number of problems. These examples provide further evidence towards the message that \emph{exactness can be treated systematically whenever $\Gamma$, $\Gamma_P$, or $\Gamma^\circ$ is well-understood}.

\subsection{Mixed binary programming}
\label{sec:mixed_binary_programming}

To begin, we apply our results to a well-studied prototypical set involving a convex quadratic function, a binary variable and a big-M relation. The example in this subsection highlights the systematic nature of our approach.

Consider the epigraph set
\begin{align*}
\cS = \set{(x,t)\in\R^2\times \R:\, \begin{array}
	{l}
	q_\obj(x) \coloneqq x_2^2 \leq 2t\\
	q_1(x) \coloneqq x_1(x_1 - 1) = 0\\
	q_2(x) \coloneqq \sqrt{2}x_2(x_1 - 1) = 0
\end{array}}.
\end{align*}
In words, $x_1$ is a binary on-off variable, $x_2$ is a continuous variable which is constrained to be off whenever $x_1$ is off, and $t$ is the epigraph variable corresponding to $x_2^2$.
The normalization of $q_2(x)$ is not important here and is made only for notational convenience in the calculations.

It is well-known that $\conv(\cS)$ is given by the perspective reformulation of $\cS$ (see e.g.,~\cite{frangioni2006perspective,gunluk2010perspective}), i.e.,
\begin{align}
	\label{eq:perspective_reformulation}
\conv(\cS) = \set{(x,t)\in\R^2\times \R:\, x_2^2 - 2 t x_1 \leq 0,\, 0\leq x_1\leq 1 }.
\end{align}

We give an alternative proof of \eqref{eq:perspective_reformulation}. We will show that $\conv(\cS) = \cS_\SDP$, the projected SDP relaxation, using \cref{thm:conv_nec}. Then, using an explicit description of $\Gamma^\circ$, we will give a description of $\conv(\cS)=\cS_\SDP$ in the original space.

A simple computation shows that in this setting, we have
\begin{gather*}
\Gamma = \set{(\gamma_\obj,\gamma)\in\R^3:\, \gamma_\obj + \gamma_1 \geq \sqrt{(\gamma_\obj-\gamma_1)^2 +\left(\sqrt{2}\gamma_2\right)^2}} \text{ and}\\
\Gamma^\circ = \set{(\ell_\obj,\ell)\in\R^3:\, -\ell_\obj-\ell_1 \geq  \sqrt{(\ell_\obj-\ell_1)^2 +\left(\sqrt{2}\ell_2\right)^2}}.
\end{gather*}
In words, $\Gamma$ and $\Gamma^\circ$ are both (rotated) second-order cones and \cref{as:definite,as:facially_exposed} hold.

It remains to show that for all $(x,t)\in\cS_\SDP\setminus \cS$, the set $\cR'(x,t)$ is nontrivial.
To this end, let $(x,t)\in\cS_\SDP\setminus \cS$.
Recall that $\Gamma^\circ$ has three types of faces: the two trivial faces (the apex and the cone itself) and the one-dimensional proper faces. Thus, there are three cases to consider: (i) $\cG(x,t) = \set{0}$, (ii) $\cG(x,t) = \Gamma^\circ$, and (iii) $\cG(x,t)$ is a one-dimensional face of $\Gamma^\circ$.

In case (i), $q(x) - 2te_\obj = 0$ implying that $(x,t)\in\cS$, a contradiction. In case (ii), $\spann(\cG(x,t)) = \R^3$ so that $\cR'(x,t) = \R^3$ and is nontrivial.
In the final case, a mechanical but slightly tedious application of \cref{prop:R'_exposed} (see \cref{sec:deferred_proofs_conv_appl}) gives
\begin{align}
	\label{eq:R_description_mixed_binary}
	\cR'(x,t) &= \set{\begin{pmatrix}
		2t\\
		-x_2\\
		0
	\end{pmatrix},\, \begin{pmatrix}
		-x_2\\
		x_1\\
		0
	\end{pmatrix},\, \begin{pmatrix}
		t\\
		-x_2\\
		x_1
	\end{pmatrix},\, \begin{pmatrix}
		x_2(x_1-1+2t)\\
		-x_1^2 + x_1 -2 t x_1 -2t\\
		2x_2
	\end{pmatrix} }^\perp.
	\end{align}
Finally, one may verify that $(x,t)\in\cR'(x,t)$ is nonzero.

\begin{remark}
Here, the motivation for the final step of checking that $(x,t)\in\cR'(x,t)$ is as follows: 
One can show that in case (iii), the first three vectors in \eqref{eq:R_description_mixed_binary} span the $2$-dimensional subspace orthogonal to $(x,t)$. In particular, 
$\cR'(x,t)$ is nontrivial if and only if $(x,t)\in\cR'(x,t)$.
\mathprog{\qed}
\end{remark}

We conclude that
\begin{align*}
&\conv(\cS) = \cS_\SDP =\\
&\quad\set{(x,t)\in\R^3:\, -(q_\obj(x) - 2t) - q_1(x) \geq \sqrt{(q_\obj(x) - 2t -q_1(x))^2 + 2q_2(x)^2}}.
\end{align*}

This example highlights the systematic nature of the approach outlined in \cref{thm:conv_nec} for proving convex hull exactness.
In contrast to \emph{ad hoc} proofs of convex hull exactness which may rely on \emph{guessing and verifying} a nonzero rounding direction, the system of equations defining $\cR'(x,t)$ gives a principled way of \emph{deducing} a direction.
While guessing such a rounding direction may be possible in low-dimensional settings (for example, the setting of the current subsection), this becomes more difficult in higher-dimensional settings where $\cS$ and $\cS_\SDP$ are difficult to visualize. We illustrate this in the following subsection.

 \subsection{Quadratic matrix programs}
\label{subsec:quadratic_matrix_programming}

Quadratic matrix programs (QMPs)~\cite{beck2007quadratic,wang2020tightness} are a generalization of QCQPs where the decision variable $x\in\R^n$ is replaced by a decision matrix $X\in\R^{r\times k}$. These problems find a variety of application and have been used to model robust least squares problems, the orthogonal Procrustes problem \cite{beck2007quadratic}, and certain sphere packing problems \cite{beck2012new}.
Formally, a QMP is an optimization problem in the variable $X\in\R^{r\times k}$, where the constraints and objective function are each of the form
\begin{align*}
\tr(X^\top \bb A X) + 2\tr(B^\top X) + c
\end{align*}
for some $\bb A\in\S^r$, $B\in\R^{r\times k}$, and $c\in\R$.

Alternatively, letting $x\in\R^n$ (resp.\ $b\in\R^n$) denote the vector formed by stacking the columns of $X$ (resp.\ $B$) on top of each other, we can rewrite the above expression as
\begin{align*}
x^\top (I_k\otimes \bb A)x + 2\ip{b, x} + c.
\end{align*}
We will choose to view QMPs as the special class of QCQPs where the quadratic forms $A_\obj,A_1,\dots,A_m$ are each of the form $I_k\otimes \bb A$ for some $\bb A \in\S^r$.

The following lemma establishes that if the number of constraints is small compared to $k$ (originally the width of the matrix variable), then convex hull exactness holds.

\begin{proposition}
\label{prop:qmp}
Suppose \cref{as:definite} holds and that $A_\obj = I_k \otimes \bb A_\obj$, $A_1=I_k \otimes \bb A_1$, \dots, $A_m =I_k \otimes \bb A_m$ for some $\bb A_\obj,\bb A_1,\dots,\bb A_m\in\S^r$. Furthermore, suppose $k\geq m$.
Then, $\cR(x,t)$ is nontrivial for every $(x,t)\in\cS_\textup{SDP}\setminus\cS$. In particular, convex hull exactness holds, i.e., $\conv(\cS) = \cS_\SDP$.
\end{proposition}
\begin{proof}
Fix $(x,t)\in\cS_\SDP\setminus\cS$.
Based on \cref{thm:conv_hull_sufficient,lem:R'_if_epigraph_tight}, our goal is to prove that
\begin{align}
\label{eq:qmp_R}
\cR'(x,t) =
\set{(x',t')\in\R^{n+1}:\, \begin{array}
	{l}
	x'^\top A(\gamma_\obj,\gamma) x' = 0 ,\,\forall (\gamma_\obj,\gamma)\in\cG(x,t)^\perp\\
	\ip{A(\gamma_\obj,\gamma)x + b(\gamma_\obj,\gamma), x'} - \gamma_\obj t' = 0 ,\,\forall (\gamma_\obj,\gamma)\in\cG(x,t)^\perp
\end{array}}
\end{align}
is nontrivial.
We claim that it suffices to show how to construct a nonzero $y\in\R^r$ such that
\begin{align}\label{eq:qmp_y}
y^\top \bb A(\gamma_\obj,\gamma) y = 0 ,\,\forall (\gamma_\obj,\gamma) \in\cG(x,t)^\perp.
\end{align}
To see that this suffices, note that for any $w\in\R^k$, the vector $x'\coloneqq w\otimes y$ satisfies the first constraint in \eqref{eq:qmp_R} since for $(\gamma_\obj,\gamma)\in\cG(x,t)^\perp$, we have
\begin{align*}
(w\otimes y)^\top A(\gamma_\obj,\gamma) (w\otimes y) = (w^\top w)(y^\top \bb A(\gamma_\obj,\gamma) y) = 0.
\end{align*}
Then, $(w\otimes y, t')\in\cR'(x,t)$ if and only if
\begin{align*}
\ip{A(\gamma_\obj,\gamma) x + b(\gamma_\obj,\gamma), w\otimes y} - \gamma_\obj t' = 0 ,\,\forall (\gamma_\obj,\gamma)\in\cG(x,t)^\perp.
\end{align*}
This is a system of
$\dim(\cG(x,t)^\perp)$-many homogeneous linear equations in the variables $(w,t)\in\R^{k+1}$.
Note that as $\cG(x,t) \ni q(x) - 2te_\obj$, which is nonzero by assumption, we have that $\dim(\cG(x,t)^\perp) \leq m$. As $k+1> m$ by assumption, we deduce that this system has a nontrivial solution. Thus, we conclude that \eqref{eq:qmp_R} is nontrivial if there exists a nonzero $y\in\R^r$ satisfying \eqref{eq:qmp_y}. 

It remains to construct $y$.
By definition of $\cS_\SDP$, there exists $Y\succeq 0$ such that
\begin{align}
\label{eq:system_in_lifted_space}
\begin{cases}
q_\obj(x) + \ip{A_\obj, Y} \leq 2t,\\
q_i(x) + \ip{A_i, Y} \leq 0,\,\forall i\in[m_I],\text{ and}\\
q_i(x) + \ip{A_i, Y} = 0 ,\,\forall i\in[m_I + 1, m].
\end{cases}
\end{align}
Without loss of generality, $Y = ({1\over k}I_k) \otimes \bb Y$.
As $(x,t)\notin\cS$, we have that $\bb Y\in\S^r_+\setminus\set{0}$ and we may pick a nonzero $y\in\R^r$ such that $yy^\top \preceq \bb Y$.
For notational convenience, let $\ell_\obj \coloneqq y^\top \bb A_\obj y$ and $\ell_i \coloneqq y^\top \bb A_i y$ for $i\in[m]$. 
Note that for any $(\gamma_\obj,\gamma)\in\Gamma$, we have $A(\gamma_\obj,\gamma)\succeq 0$, or equivalently $\bb A(\gamma_\obj,\gamma)\succeq 0$. Thus, $yy^\top \preceq \bb Y$ implies that 
$\ip{(\gamma_\obj, \gamma), (\ell_\obj,\ell)} 
=y^\top \bb A(\gamma_\obj,\gamma) y 
\leq \ip{\bb A(\gamma_\obj,\gamma),\bb Y}$. Also, from $Y = ({1\over k}I_k) \otimes \bb Y$ and the relation between the matrices $A_\obj,A_i$ and $\bb A_\obj, \bb A_i$, we have $\ip{\bb A(\gamma_\obj,\gamma),\bb Y}=\ip{A(\gamma_\obj,\gamma),Y}$. 
We deduce that for $(\gamma_\obj,\gamma)\in\Gamma$,
\begin{gather*}
\ip{\begin{pmatrix}
	\gamma_\obj\\
	\gamma
\end{pmatrix},q(x) - 2te_\obj + \begin{pmatrix}
	\ell_\obj\\\ell
\end{pmatrix}} \leq 
\ip{\begin{pmatrix}
	\gamma_\obj\\
	\gamma
\end{pmatrix},
q(x) - 2te_\obj + \begin{pmatrix}
	\ip{A_\obj, Y}\\
	\left(\ip{A_i, Y}\right)_i
\end{pmatrix}} \leq 0,\end{gather*}
where the last inequality follows from \eqref{eq:system_in_lifted_space} and $(\gamma_\obj,\gamma)\in\Gamma$.
This then shows that $q_\obj(x) - 2te_\obj + (\ell_\obj,\ell)\in\Gamma^\circ$. Moreover, because $\bb A(\gamma_\obj,\gamma)\succeq 0$, we have 
\begin{gather*}
0 \geq - y^\top \bb A(\gamma_\obj,\gamma) y 
= \ip{\begin{pmatrix}
	\gamma_\obj\\\gamma
\end{pmatrix}, \begin{pmatrix}
	-\ell_\obj\\
	-\ell
\end{pmatrix}},
\end{gather*}
which implies $-(\ell_\obj,\ell)\in\Gamma^\circ$. 
We have shown that $q_\obj(x) - 2te_\obj + (\ell_\obj,\ell)$ and $-(\ell_\obj,\ell)$ both lie in $\Gamma^\circ$. Then, as $q_\obj(x) -2te_\obj\in\rint(\cG(x,t))$, we deduce that $(\ell_\obj,\ell)\in\spann(\cG(x,t))$.
In particular, $y^\top \bb A(\gamma_\obj,\gamma) y = \ip{(\gamma_\obj, \gamma), (\ell_\obj,\ell)} = 0$ for all $(\gamma_\obj,\gamma)\in\cG(x,t)^\perp$.
\end{proof}

\begin{remark}
\label{rem:qmp}
SDP exactness in the context of QMPs was previously studied by \citet{wang2020tightness,beck2007quadratic,beck2012new}.
Specifically, \citet{beck2007quadratic} shows that objective value exactness holds whenever $k\geq m$ and \citet{wang2020tightness} show that convex hull exactness holds whenever $k \geq m + 2$.
\cref{prop:qmp} strengthens both of these results by showing that convex hull exactness holds whenever $k \geq m$.\mathprog{\qed}
\end{remark}

 \subsection{The partition problem}
\label{sec:partition_problem}
We next consider the partition QCQP and its SDP relaxation.
Recall the partition QCQP: Given $a\in\R^n$, we want to minimize
\begin{align*}
\Opt \coloneqq \min_{x\in\R^n}\set{(a^\top x)^2:\, x_i^2 =1 ,\,\forall i\in[n]}.
\end{align*}
Note that $\Opt = 0$ if and only if the vector $a$ can be \emph{partitioned} into two sets of equal weight.
Thus, deciding whether $\Opt = 0$ is NP-hard \cite{karp1972reducibility}.
In this section, we will first give an explicit description of $\cS_\SDP$ under a minor assumption. This explicit $\cS_\SDP$ description will then let us conclude that $\conv(\cS)\neq \cS_\SDP$ under the same minor assumption.

\begin{assumption}
\label{as:partition_positive_weights}
	$a\in\R^n_{++}$ and $n\geq 2$.
\end{assumption}
\begin{remark}
\cref{as:partition_positive_weights} is essentially without loss of generality: It is straightforward to derive a closed form description of $\cS_\SDP$ when $n = 1$. 
Similarly, one can relate $\cS_\SDP$ corresponding to an arbitrary $a\in\R^n$ with the set $\cS_\SDP$ corresponding to some $a'\in\R^{n'}_{++}$ for $n'\coloneqq \abs{\set{i\in[n]:\,a_i\neq0}}$.\mathprog{\qed}
\end{remark}

\begin{restatable}{proposition}{lempartitionexplicit}
	\label{lem:partition_explicit}
	Suppose \cref{as:partition_positive_weights} holds. Then,
	\begin{gather*}
	\Gamma = \set{(\gamma_\obj,\gamma)\in\R\times\R^n:\, \begin{array}
		{l}
		\gamma_\obj aa^\top + \Diag(\gamma)\succeq 0\\
		\gamma_\obj \geq 0
	\end{array}}, \text{ and}\\
	\cS_\SDP = \set{(x,t)\in[-1,1]^n\times\R:~
	(a^\top x)^2 + \max_{i\in[n]} \left(a_i\sqrt{1-x_i^2} - \sum_{j\neq i}a_j\sqrt{1-x_j^2}\right)_+^2 \leq 2t
	} .
	\end{gather*}
\end{restatable}
See \cref{sec:deferred_proofs_conv_appl} for a proof of this statement.

Recall from \cite{laurent1995positive} that a vector $a\in\R^n_{++}$ is said to be \emph{balanced} if for all $i\in[n]$, $a_i \leq \sum_{j\neq i} a_j$.
The following result then follows as a corollary to \cref{lem:partition_explicit}. (See \cref{sec:deferred_proofs_conv_appl}.)
\begin{restatable}
	{corollary}{partitionsdpzero}\label{cor:partition_sdp=0}
	Suppose \cref{as:partition_positive_weights} holds. Then, $\Opt_\textup{SDP} = 0$ if and only if $a$ is balanced.
\end{restatable}

As a consequence of \cref{cor:partition_sdp=0} (and the NP-hardness of deciding whether $\Opt = 0$ for the partition QCQP), we see that it is NP-hard to decide whether objective value exactness holds for the partition QCQP. This recovers a result due to \citet{laurent1995positive}.

In contrast to the NP-hardness of checking \textit{objective value exactness} for the partition QCQP, the following corollary states that checking \textit{convex hull exactness} for the partition QCQP is a trivial task.

\begin{restatable}
	{corollary}{thmpartitionconvneq}
	\label{thm:partition_conv_neq}
	Suppose \cref{as:partition_positive_weights} holds. Then, $\conv(\cS)\neq\cS_\SDP$.
\end{restatable}
The proof of \cref{thm:partition_conv_neq} follows from the observation that $\conv(\cS)$ is polyhedral and that $\cS_\SDP$ is not polyhedral. See \cref{sec:deferred_proofs_conv_appl} for details.

\section{Objective value exactness}
\label{sec:obj_val_exactness}

In this section, we present sufficient conditions for objective value exactness, i.e., the property that $\Opt = \Opt_\SDP$.
In fact, all of our sufficient conditions imply the stronger condition, which we refer to as \emph{optimizer exactness}, that the optimizers of the QCQP and its SDP relaxation coincide, i.e.,
\begin{align*}
\argmin_{(x,t)\in\cS_\SDP} 2t
=
\argmin_{(x,t)\in\cS} 2t.
\end{align*}
We begin by presenting sufficient conditions stemming from a primal analysis. These sufficient conditions generalize \cite[Theorem 3]{wang2020tightness}. Our second set of sufficient conditions are based on a dual analysis and require the additional assumption that the dual optimum is achieved. These conditions imply further that the optimizers are unique.

\subsection{Sufficient conditions based on a primal analysis}
\label{subsec:obj_primal}

We begin by presenting a very general sufficient condition for optimizer exactness.
\begin{theorem}
\label{thm:suff_obj_primal}
Suppose \cref{as:definite} holds. Furthermore, suppose that for all $(x,t)\in\cS_\SDP\setminus \cS$, there exists closed cones $K_1,K_2\subseteq \R^{1+m}$ and $(x',t')\in\R^{n+1}$ satisfying
\begin{align}
\label{eq:suff_cond_obj_primal}
\begin{cases}
	K_1\subseteq (q(x) - 2te_\obj)^\perp\\
	-K_2 \cap (q(x) - 2te_\obj)^\circ = \set{0}\\
	K_1 + K_2 \supseteq \Gamma\\
	(x')^\top A(\gamma_\obj, \gamma) x' = 0 ,~\forall (\gamma_\obj,\gamma)\in K_1\\
	\ip{A(\gamma_\obj,\gamma)x + b(\gamma_\obj,\gamma), x'} - \gamma_\obj t' \leq 0,~\forall (\gamma_\obj,\gamma)\in K_1\\
	t' <0
\end{cases} .
\end{align}
Then, optimizer exactness holds, i.e., $\argmin_{(x,t)\in\cS_\SDP} 2t = \argmin_{(x,t)\in\cS} 2t$.
\end{theorem}
\begin{proof}
Let $(x,t)\in\cS_\SDP\setminus\cS$. It suffices to show that $(x,t)\notin\argmin_{(x,t)\in\cS_\SDP}2t$. Let $K_1,K_2,x',t'$ denote the quantities furnished by the assumption.

We claim that for all $\alpha >0$ small enough, $(x+\alpha x',t+\alpha t')\in\cS_\SDP$. Indeed, for all $\alpha>0$ small enough and $(\gamma_\obj,\gamma)\in K_1$, 
\begin{align*}
&\ip{(\gamma_\obj,\gamma), q(x+\alpha x') - 2 (t+\alpha t')e_\obj}\\
&\qquad = \alpha^2 \underbrace{(x')^\top A(\gamma_\obj,\gamma) x'}_{=0} +
2\alpha \underbrace{\left(\ip{A(\gamma_\obj,\gamma)x + b(\gamma_\obj,\gamma), x'} - \gamma_\obj t'\right)}_{\leq0}\\
&\qquad\quad +\underbrace{\ip{(\gamma_\obj,\gamma), q(x) -2te_\obj}}_{=0}\\
&\qquad \leq 0.
\end{align*}

Next, set $\cB\coloneqq K_2\cap \bS^{(1+m)-1}$ so that $\cone(\cB) = K_2$. By definition of $K_2$ and $\cB$, we have $-\cB\cap (q(x) - 2te_\obj)^\circ = \emptyset$ so that the map
\begin{align*}
\alpha \mapsto \max_{(\gamma_\obj,\gamma)\in\cB} \ip{(\gamma_\obj,\gamma), q(x+\alpha x') - 2(t+\alpha t')e_\obj}
\end{align*}
is negative at $\alpha = 0$. Note also that this map is a continuous function of $\alpha$. Then, by continuity, this map is negative for some $\alpha>0$.

Finally, by linearity and the fact that $K_1 + K_2\supseteq \Gamma$, we deduce that $(x+\alpha x',t+\alpha t')\in\cS_\SDP$ for some $\alpha>0$. This shows $(x,t)\notin\argmin_{(x,t)\in\cS_\SDP}2t$.
\end{proof}

We next recover more concrete sufficient conditions by picking $K_1$ and $K_2$ appropriately. The following corollary recovers the sufficient condition for objective value exactness (in the setting of polyhedral $\Gamma$) presented in \cite[Theorem 3]{wang2020tightness}.
\begin{corollary}
\label{cor:polyhedral_suff_primal_obj}
Suppose \cref{as:definite} holds and that $\Gamma$ is polyhedral. Furthermore, suppose that for all $(x,t)\in\cS_\SDP\setminus\cS$, there exists $(x',t')\in\R^{n+1}$ satisfying
\begin{align}
\label{eq:polyhedral_suff_primal_obj}
\begin{cases}
	(x')^\top A(\gamma_\obj,\gamma) x' = 0,~\forall (\gamma_\obj,\gamma)\in\cF(x,t)\\
	\ip{b(\gamma_\obj,\gamma),x'} - \gamma_\obj t'\leq 0,~\forall (\gamma_\obj,\gamma)\in\cF(x,t)\\
	t'<0
\end{cases}.
\end{align}
Then, $\argmin_{(x,t)\in\cS}2t = \argmin_{(x,t)\in\cS_\SDP}2t$.
\end{corollary}
\begin{proof}
Let $(x,t)\in\cS_\SDP\setminus\cS$. Since $\Gamma$ is polyhedral, we can write $\Gamma\coloneqq\cone\set{(\gamma_\obj^{(i)}, \gamma^{(i)})}_{i\in [T]}$ for a finite set of generators.
Take,
\begin{align*}
K_1 = \cone\set{(\gamma_\obj^{(i)},\gamma^{(i)}):\, \ip{(\gamma_\obj^{(i)},\gamma^{(i)}), q(x) - 2t e_\obj} = 0} = \cF(x,t)
\end{align*}
and
\begin{align*}
K_2 &= \cone\set{(\gamma_\obj^{(i)},\gamma^{(i)}):\, \ip{(\gamma_\obj^{(i)},\gamma^{(i)}), q(x) - 2t e_\obj} < 0}.
\end{align*}
Note that $K_1$ and $K_2$ are polyhedral and thus closed. Moreover, the first three requirements of \eqref{eq:suff_cond_obj_primal} are satisfied for this choice of $K_1$ and $K_2$. Moreover, note that for every $(\gamma_\obj,\gamma)\in\cF(x,t)\subseteq\Gamma$ we have $A(\gamma_\obj,\gamma)\succeq0$ and for any $A\succeq0$, $x^\top Ax=0$ implies $Ax=0$. Thus, from $K_1=\cF(x,t)$, we deduce $\ip{A(\gamma_\obj,\gamma)x, x'}=0$ for every $(\gamma_\obj,\gamma)\in K_1$ so that the last three requirements of \eqref{eq:suff_cond_obj_primal} coincide with \eqref{eq:polyhedral_suff_primal_obj}.
\end{proof}

The following corollary derives a sufficient condition for objective value exactness without the assumption that $\Gamma$ is polyhedral. In words, this assumption supposes that for any $(x,t)\in\cS_\SDP\setminus\cS$, there exists a direction $(x',t')\in\R^{n+1}$ such that $q(x+\alpha x') - 2(t+\alpha t')e_\obj$ varies only along the line containing $q(x) - 2te_\obj$. In particular, by picking $\alpha$ appropriately, we can achieve $q(x+\alpha x') - 2(t+\alpha t')e_\obj = 0$. 

\begin{corollary}
\label{cor:one_shot_obj}
Suppose \cref{as:definite} holds. Furthermore, suppose that for all $(x,t)\in\cS_\SDP\setminus \cS$, there exists $(x',t')\in\R^{n+1}$ satisfying
\begin{align}\label{eq:one_shot_obj}
\begin{cases}
	(x')^\top A(\gamma_\obj, \gamma) x' = 0 ,~\forall (\gamma_\obj,\gamma)\in (q(x) - 2te_\obj)^\perp\\
	\ip{A(\gamma_\obj,\gamma)x + b(\gamma_\obj,\gamma), x'} - \gamma_\obj t' = 0,~\forall (\gamma_\obj,\gamma)\in (q(x) - 2te_\obj)^\perp\\
	t' <0
\end{cases} .
\end{align}
Then, $\argmin_{(x,t)\in\cS_\SDP} 2t = \argmin_{(x,t)\in\cS} 2t$.
\end{corollary}
\begin{proof}
Take $K_1 = (q(x) - 2te_\obj))^\perp$ and $K_2 = -\cone(q(x) - 2te_\obj)$. Then, $K_1$ and $K_2$ are both closed convex cones and we can easily observe that the first three requirements in \eqref{eq:suff_cond_obj_primal} are automatically satisfied for this choice of $K_1$ and $K_2$. The last three requirements in \eqref{eq:suff_cond_obj_primal} coincide with \eqref{eq:one_shot_obj}.
\end{proof}

\subsection{Sufficient conditions based on a dual analysis}
\label{subsec:obj_dual}
Next, we give a strengthened sufficient condition for objective value exactness depending on a dual analysis.
To this end, we define the concave extended-real valued function $\mathbf{d}:\R^m\to\R\cup\set{-\infty}$ by
\begin{align*}
\mathbf{d}(\gamma)\coloneqq \inf_{x\in\R^n} [\gamma,q(x)].
\end{align*}
\begin{remark}
Recall here that by \cref{cor:opt_sdp_saddle}, we can write $\Opt_\SDP$ in the saddle-point form $\Opt_\SDP = \inf_{x\in\R^n}\sup_{\gamma\in\Gamma_p} [\gamma,q(x)]$ given in \eqref{eq:SDP_SP}. Whence, by coercivity~\cite[Proposition VI.2.3]{ekeland1999convex} we can also write 
$\Opt_\SDP = \sup_{\gamma\in\Gamma_p}\mathbf{d}(\gamma)$.\mathprog{\qed}
\end{remark}

The following theorem states that if $\mathbf{d}(\gamma)$ is maximized at a point $\gamma$ where $A[\gamma]\succ 0$ (e.g., on the interior of $\Gamma_P$), then optimizer exactness holds. This theorem can be interpreted as the observation that if the dual to \eqref{eq:sdp_relaxation_primal} in $\S^{n+1}$ has a rank-$n$ optimizer, then \eqref{eq:sdp_relaxation_primal} has a unique rank-$1$ solution. This is well-known and has been vastly explored in the literature. We state it as a theorem not because it is new or difficult to prove but because of its importance in deriving additional sufficient conditions (see \cref{cor:sufficiently_steep,cor:suff_obj_dual}).

\begin{theorem}
\label{lem:suff_obj_dual}
Suppose \cref{as:definite} holds and that $\sup_{\gamma\in\Gamma_P}\mathbf{d}(\gamma)$ is achieved at some $\gamma^*$ for which $A[\gamma^*]\succ 0$ (e.g., $\gamma^*\in\inter(\Gamma)$). Then, $\argmin_{(x,t)\in\cS}2t = \argmin_{(x,t)\in\cS_\SDP}2t$. Furthermore, the minimizers of these two optimization problems are unique.
\end{theorem}
\begin{proof}
It suffices to show that $\min_{(x,t)\in\cS_\SDP}2t$ has a unique solution $(x^*,t^*)$ and that $(x^*,t^*)\in\cS$.
Let $(x^*,t^*)\in\argmin_{(x,t)\in\cS_\SDP}2t$
so that $x^*\in\argmin_x\sup_{\gamma\in\Gamma_P} [\gamma,q(x)]$ and $2t^* = \sup_{\gamma\in\Gamma_P}[\gamma,q(x^*)]$.
By the Saddle Point Theorem applied to \eqref{eq:SDP_SP}, we have
\begin{align*}
0_n = \grad_x [\gamma^*,q(x^*)] = 2\left(A[\gamma^*]x^* + b[\gamma^*]\right).
\end{align*}
Because $A[\gamma^*]\succ 0$, which implies that $x^* = -A[\gamma^*]^{-1}b[\gamma^*]$. This proves uniqueness of $(x^*,t^*)$.

Note that there exists $\alpha>0$ such that $[\gamma^*, \gamma^* + \alpha e_i]\subseteq \Gamma_P$ for all $i\in[m_I]$ and $[\gamma^*\pm \alpha  e_i]\subseteq\Gamma_P$ for all $i\in[m_I+1,m]$.
Then, by the Saddle Point Theorem we have
\begin{gather*}
q_i(x^*) = \grad_{\gamma_i}[\gamma^*,q(x^*)] \leq 0,\,\forall i\in[m_I],\text{ and}\\
q_i(x^*) = \grad_{\gamma_i}[\gamma^*,q(x^*)] = 0,\,\forall i\in[m_I+1,m].
\end{gather*}
We deduce that $q_\obj(x^*) \leq \sup_{\gamma\in\Gamma_P}[\gamma,q(x^*)] = 2t^*$. Hence, we conclude that $(x^*,t^*)\in\cS$.
\end{proof}

\begin{remark}\label{rem:inter-d_gamma}
Note that for any $\gamma$ for which $A[\gamma]\succ 0$, the dual function $\mathbf{d}(\gamma)$ is the sum of a linear function $c[\gamma]$ and a concave function $-b[\gamma]^\top A[\gamma]^{-1}b[\gamma]$, i.e.,
\begin{align*}
\mathbf{d}(\gamma) = -b[\gamma]^\top A[\gamma]^{-1}b[\gamma] + c[\gamma].
\end{align*}
We will use this structure to derive more concrete sufficient conditions ensuring that $\mathbf{d}(\gamma)$ is maximized at some point $\gamma^*$ for which $A[\gamma^*]\succ 0$. \mathprog{\qed}
\end{remark}

The following sufficient condition can be interpreted as requiring $-b[\gamma]^\top A[\gamma]^{-1}b[\gamma]$ (and hence also $\mathbf{d}(\gamma)$) to diverge to $-\infty$ as $\gamma$ approaches a point $\hat\gamma\in\Gamma_P$ for which $A[\hat\gamma]\not\succ 0$.
\begin{corollary}
\label{cor:suff_obj_dual}
	Suppose \cref{as:definite} holds and that $\sup_{\gamma\in\Gamma_P}\mathbf{d}(\gamma)$ is achieved. Furthermore, suppose that for all $\gamma\in\Gamma_P$, we have
	\begin{align*}
	A[\gamma]\not\succ 0 \quad\implies\quad \exists v\in\ker(A[\gamma])\text{ s.t. }  \ip{v,b[\gamma]\neq 0}.
	\end{align*}
	Then, $\argmin_{(x,t)\in\cS}2t = \argmin_{(x,t)\in\cS_\SDP}2t$. Furthermore, the minimizers of these two optimization problems are unique.
\end{corollary}
\begin{proof}
Let $\gamma^*\in\argmax_{\gamma\in\Gamma_P}\mathbf{d}(\gamma)$. 
By \cref{lem:suff_obj_dual}, it suffices to show that
$A[\gamma^*]\succ 0$.
Suppose otherwise, so that $A[\gamma^*]\not\succ 0$.
Then, the assumptions of the corollary furnish a $v\in\ker(A[\gamma^*])$ such that $\ip{v, b[\gamma^*]} \neq 0$.
Let $(x^*,t^*)\in\argmin_{(x,t)\in\S_\SDP}2t$. By the Saddle Point Theorem applied to \eqref{eq:SDP_SP}, we deduce
\begin{align*}
0 = \ip{v,0_n} = \ip{v,\grad_x [\gamma^*, q(x^*)]} = 2\ip{v,A[\gamma^*]x^* + b[\gamma^*]} \neq 0,
\end{align*}
a contradiction.
\end{proof}
\begin{remark}
\citet{burer2019exact} study diagonal QCQPs and show~\cite[Theorem 1]{burer2019exact} that objective value exactness holds whenever certain systems of equations are infeasible. Specifically, their sufficient condition for diagonal QCQPs can be rewritten as the condition that for any $i\in[n]$, the system $\set{\gamma\in\Gamma_p,\,e_i\in\ker(A[\gamma]),\, b[\gamma]_i= 0}$ is infeasible. \cref{cor:suff_obj_dual}  generalizes \cite[Theorem 1]{burer2019exact} by considering general matrices $A_i$ as opposed to diagonal matrices considered in \cite{burer2019exact}.
\mathprog{\qed}
\end{remark}

Alternatively, one may impose the slightly weaker condition that $-b[\gamma]^\top A[\gamma]^{-1}b[\gamma]$ gets ``sufficiently steep near points at which $A[\gamma]\not\succ 0$'' compared to $\norm{(c_1,\dots,c_m)}_2$.
\begin{corollary}
\label{cor:sufficiently_steep}
Suppose \cref{as:definite} holds and that $\sup_{\gamma\in\Gamma_P}\mathbf{d}(\gamma)$ is achieved. Furthermore, suppose that for all $\gamma\in\Gamma_P$ such that $A[\gamma]\not\succ 0$, there exists $\delta\in\R^m$ such that $\gamma_\delta\coloneqq \gamma + \delta\in\inter(\Gamma_P)$ and $\gamma_{2\delta}\coloneqq \gamma+2\delta\in\inter(\Gamma_P)$ and 
\begin{align*}
\left(-b[\gamma_\delta]^\top A[\gamma_\delta]^{-1} b[\gamma_\delta]\right) - \left(-b[\gamma_{2\delta}]^\top A[\gamma_{2\delta}]^{-1} b[\gamma_{2\delta}]\right) \leq -\norm{\delta}_2 \sqrt{\sum_{i=1}^m c_i^2}.
\end{align*}
Then, $\argmin_{(x,t)\in\cS}2t = \argmin_{(x,t)\in\cS_\SDP}2t$. Furthermore, the minimizers of these two optimization problems are unique.
\end{corollary}
\begin{proof}
Let $\gamma^*\in\argmax_{\gamma\in\Gamma_P}\mathbf{d}(\gamma)$. We will construct an optimizer $\tilde \gamma\in \argmax_{\gamma\in\Gamma_P}\mathbf{d}(\gamma)$ for which $A[\tilde\gamma]\succ 0$.
The result will then follow from \cref{lem:suff_obj_dual}.

If $A[\gamma^*]\succ 0$ then we may take $\tilde \gamma = \gamma^*$.
Else, let $\delta$ be furnished by the assumption of the corollary and note that $\delta\neq 0$.
We will set $\tilde \gamma = \gamma^*_\delta$. 
Then, $\tilde \gamma \in\inter(\Gamma)$ and thus $A[\tilde \gamma]\succ0$. 
By optimality of $\gamma^*$, it suffices to show that $\mathbf{d}(\gamma^*)\leq \mathbf{d}(\gamma^*_\delta)$. As $\mathbf{d}(\gamma)$ is concave and $\gamma^*, \gamma^*_\delta,\gamma^*_{2\delta}$ lie on a line, it suffices in turn to show that $\mathbf{d}(\gamma^*_\delta) \leq \mathbf{d}(\gamma^*_{2\delta})$. Finally, as $\gamma^*_\delta$ and $\gamma^*_{2\delta}$ both lie in $\inter(\Gamma_P)$, we may expand
\begin{align*}
\mathbf{d}(\gamma^*_\delta) - \mathbf{d}(\gamma^*_{2\delta}) &=
\left(-b[\gamma^*_\delta]A[\gamma^*_\delta]^{-1}b[\gamma^*_\delta] + c[\gamma^*_\delta]\right) -
\left(-b[\gamma^*_{2\delta}]A[\gamma^*_{2\delta}]^{-1}b[\gamma^*_{2\delta}] + c[\gamma^*_{2\delta}]\right)\\
&\leq -\norm{\delta}_2 \sqrt{\sum_{i=1}^m c_i^2}  - \sum_{i=1}^m \delta_i c_i\leq 0.
\end{align*}
Applying \cref{lem:suff_obj_dual} concludes the proof.
\end{proof}

\section{Applications: Objective value exactness}
\label{sec:obj_appl}
In this section, we apply the results of \cref{sec:obj_val_exactness} to random and semi-random QCQPs. Again, these examples offer further evidence that \emph{questions of exactness can be treated systematically whenever $\Gamma$, $\Gamma_P$, or $\Gamma^\circ$ is well-understood}.
In fact, the results in this section show that the ideas of \cref{sec:obj_val_exactness} can be applied (at times with additive errors) even when the dual set $\Gamma$, $\Gamma_P$, or $\Gamma^\circ$ is not known exactly, but only approximately.
The random and semi-random QCQPs considered in this section are motivated by recent work~\cite{burer2019exact,locatelli2020kkt}, which has treated random QCQPs as a testing ground for understanding the strength or explanatory power of various sufficient conditions for objective value exactness.

We will fix $m$, the number of quadratic constraints, and take $n$, the number of variables, to $+\infty$ independently. We will abbreviate ``with probability $1 - o(1)$ as $n\to+\infty$'' as ``asymptotically almost surely'' (\aas).

The random and semi-random QCQPs we will consider in this section will involve data generated according to the normalized Gaussian Orthogonal Ensemble (NGOE). We collect some basic facts on the NGOE in the following section.

\subsection{Preliminaries on the (normalized) Gaussian Orthogonal Ensemble}
Here, we recall the normalized GOE and a few of its basic properties.

\begin{definition}
Let $A\in\S^n$ be a random matrix where: each diagonal entry $A_{i,i}$ is i.i.d.\ $N(0,1/2n)$; each superdiagonal entry $A_{i,j}$ is i.i.d.\ $N(0,1/4n)$; and each subdiagonal entry $A_{i,j}$ is defined by symmetry. We will refer to this distribution as the \emph{normalized Gaussian Orthogonal Ensemble} (NGOE). We will write
\begin{align*}
A\sim \ngoe(n)
\end{align*}
to denote the fact that $A$ is drawn according to this distribution.\mathprog{\qed}
\end{definition}

\begin{remark}
A different procedure for generating the same distribution is: sample $M\in\R^{n\by n}$ with every entry i.i.d.\ $N(0,1/2n)$ and return $A = (M + M^\top)/2$.\mathprog{\qed}
\end{remark}

The NGOE is a very well-understood distribution~\cite{tao2012topics}.
We will only need a few basic facts.
The first two facts state that the NGOE is invariant under various notions of rotation.
\begin{fact}
\label{fact:goe_rotation_invariant}
Fix $U\in\R^{n\by n}$ orthogonal and let $A\sim \ngoe(n)$. Then, $U^\top A U\sim \ngoe(n)$.
\end{fact}
\begin{fact}
\label{fact:goe_reweighting_invariant}
Fix $U\in\R^{k\by k}$ orthogonal and let $A_1,\dots, A_k\stackrel{\text{i.i.d.}}{\sim} \ngoe(n)$. Define $\tilde A_i \coloneqq \sum_{j=1}^k U_{i,j} A_j$.
	Then, $\tilde A_1,\dots,\tilde A_k\stackrel{\text{i.i.d.}}{\sim}\ngoe(n)$.
\end{fact}

Define also the \emph{normalized semicircular measure}
	\begin{align*}
	\mu_\nsc \coloneqq \frac{2}{\pi}\sqrt{(1 - x^2)_+}.
	\end{align*}

The next fact states that the NGOE obeys the semicircle law.
\begin{fact}
\label{fact:goe_semicircle}
	For any $\psi\in C_c^\infty(\R)$ and $\epsilon>0$,
	\begin{align*}
	\lim_{n\to\infty} \Pr\left[\abs{\int \psi d\mu_n - \int \psi d\mu_\nsc} >\epsilon\right] = 0.
	\end{align*}
	Here, $\mu_n$ is the \emph{random} measure constructed by sampling $A\sim\ngoe(n)$ and setting
$\mu_n \coloneqq \frac{1}{n}\sum_{j=1}^n \delta_{\lambda_j(A)}$,
where $\delta_{\lambda_j(A)}$ is the Dirac measure at $\lambda_j(A)$.
\end{fact}

Finally, we recall that the operator norm of $A\sim \ngoe(n)$ is $\approx 1$ asymptotically almost surely.
\begin{fact}
\label{fact:goe_bounded_norm}
Fix $\epsilon>0$ and let $A\sim\ngoe(n)$. Then, $-\lambda_{\min}(A), \lambda_{\max}(A)\in [1\pm \epsilon]$ \aas.
\end{fact}

\subsection{Exactness in the fully Gaussian setting}
\label{subsec:fully_gaussian}
This subsection considers random Euclidean distance minimization problems of the form
\begin{align}
\label{eq:euc_dist_qcqp}
\inf_{x\in\R^n}\set{\norm{x}_2^2:\, q_i(x) = 0,\,\forall i\in[m]}.
\end{align}
In words, we are looking for minimum norm solutions to random quadratic systems.

We will sample each quadratic constraint $q_i(x) = x^\top A_i x + 2b_i^\top x + c_i$ independently where $A_i\sim\ngoe(n)$, $b_i\sim N(0,I_n/n)$, and $c_i\sim N(0,1)$.
Here, the normalization on the $A_i$s and $b_i$s are chosen so that $\norm{A_i}_2 \approx 1$ and $\norm{b_i}_2\approx 1$.

Below, we will show that for any fixed $m$ and $n\to\infty$, \eqref{eq:euc_dist_qcqp} has an exact SDP relaxation \aas. Specifically, we will apply ideas from \cref{cor:suff_obj_dual} to prove:

\begin{proposition}
\label{thm:rank_one_edm}
Let $A_1\dots, A_m\stackrel{\text{i.i.d.}}{\sim}\ngoe(n)$, $b_1,\dots,b_m\stackrel{\text{i.i.d.}}{\sim} N(0, I_n/n)$ and $c_1,\dots,c_m\stackrel{\text{i.i.d.}}{\sim} N(0,1)$ be independent. Then, \aas,
optimizer exactness holds in \eqref{eq:euc_dist_qcqp}, i.e., $\argmin_{(x,t)\in\cS}2t = \argmin_{(x,t)\in\cS_\SDP}2t$.
\end{proposition}
We will highlight the very simple geometric ideas underlying the proof of this result and defer proofs of the more technical lemmas to \cref{sec:deferred_proofs_obj_appl}.

We will \cref{thm:rank_one_edm} using \cref{lem:suff_obj_dual}; specifically, we will show that $\mathbf{d}(\gamma)$ is maximized on the interior of $\Gamma_P$. As a first step, we observe that $\Gamma_P$ contains the unit ball (shrunk by $\epsilon$) \aas. The following lemma follows from an $\epsilon$-net argument, concavity of $\lambda_{\min}(A[\gamma])$ as a function of $\gamma$, and \cref{fact:goe_bounded_norm,fact:goe_semicircle}.

\begin{lemma}
	\label{lem:fg_lambda_min_concentration}
	Fix $r\geq0$ and $\epsilon>0$. Let $A_1\dots, A_m\stackrel{\text{i.i.d.}}{\sim}\ngoe(n)$. Then, \aas,
	\begin{align*}
	\lambda_{\min}(A[\gamma]) \in [1 - r \pm \epsilon],~ \forall \gamma\in r\bS^{m-1}.
	\end{align*}
	In particular, $\inter(\Gamma_P)\supseteq B(0,1-\epsilon)$ \aas.
\end{lemma}

Recall \cref{rem:inter-d_gamma} that for $\gamma\in\inter(\Gamma_P)$, we can write
\begin{align*}
\mathbf{d}(\gamma) = - b[\gamma]^\top A[\gamma]^{-1}b[\gamma] + c[\gamma].
\end{align*}
The next lemma notes that the first term in $\mathbf{d}(\gamma)$, i.e., $-b[\gamma]^\top A[\gamma]^{-1}b[\gamma]$, concentrates to a \emph{sphere cap} and follows from \cref{fact:goe_semicircle}.
\begin{lemma}
\label{lem:fg_sphere_cap_concentration}
Fix $r\in(0,1)$ and $\epsilon>0$. Let $A_1\dots, A_m\stackrel{\text{i.i.d.}}{\sim}\ngoe(n)$. Then, \aas,
\begin{align*}
-b[\gamma]^\top A[\gamma]^{-1}b[\gamma]\in[\phi(r) \pm \epsilon],~\forall \gamma\in r\bS^{m-1},
\end{align*}
where $\phi(r) \coloneqq 2(\sqrt{1 - r^2} - 1)$.
\end{lemma}

We are now ready to prove \cref{thm:rank_one_edm}. The proof will observe that the gradient of $-b[\gamma]^\top A[\gamma]^{-1}b[\gamma]$ gets ``arbitrarily steep at the boundary of $\Gamma_P$'' so that any maximizer of $\mathbf{d}(\gamma)$ must lie in $\inter(\Gamma_P)$. One may compare the proof of \cref{thm:rank_one_edm} to \cref{cor:sufficiently_steep}.
\begin{proof}
[Proof of \cref{thm:rank_one_edm}]
For convenience, let $c\in\R^m$ denote the vector with $i$th coordinate $c_i$.

Fix $\delta>0$ and let $M>0$ such that $\Pr_{c}\left[\norm{c}_2\leq M\right] \geq 1 - \delta/2$.
Let $0<r_1<r_2<1$ and $\epsilon\in (0,1 - r_2)$ such that
\begin{align*}
\frac{\phi(r_1) - \phi(r_2) - 2\epsilon}{r_2 - r_1} \geq M.
\end{align*}

In the remainder of the proof, we will condition on the events that $\norm{c}_2 \leq M$,
\begin{gather*}
\lambda_{\min}(A[\gamma]) \geq 1 - r_2 - \epsilon ,~\forall \gamma\in r_2\bS^{m-1},\\
-b[\gamma]^\top A[\gamma]^{-1}b[\gamma] \geq \phi(r_1) - \epsilon,~\forall \gamma\in r_1\bS^{m-1},\text{ and}\\
-b[\gamma]^\top A[\gamma]^{-1}b[\gamma] \leq \phi(r_2) + \epsilon,~\forall \gamma\in r_2\bS^{m-1}.
\end{gather*}
By \cref{lem:fg_lambda_min_concentration,lem:fg_sphere_cap_concentration},
this holds with probability $1 - \delta$ for all $n$ large enough.

Let $\gamma\in \Gamma_P\setminus B(0,r_2)$ and let $\gamma^{(1)}$, $\gamma^{(2)}$ denote the projections of $\gamma$ onto $B(0,r_1)$ and $B(0,r_2)$ respectively.
We claim that $\mathbf{d}(\gamma^{(2)})\geq \mathbf{d}(\gamma)$. By concavity of $\mathbf{d}(\gamma)$, it suffices to show that $\mathbf{d}(\gamma^{(1)})\geq \mathbf{d}(\gamma^{(2)})$. We compute,
\begin{align*}
\mathbf{d}(\gamma^{(1)}) &= -b\left[\gamma^{(1)}\right]A\left[\gamma^{(1)}\right]^{-1}b\left[\gamma^{(1)}\right] + \ip{c,\gamma^{(1)}}\\
&\geq \mathbf{d}(\gamma^{(2)})
+ (\phi(r_1) - \epsilon) - (\phi(r_2) + \epsilon)
+ \ip{c,\gamma^{(1)} - \gamma^{(2)}}\\
&\geq \mathbf{d}(\gamma^{(2)})
+ (\phi(r_1) - \phi(r_2) - 2\epsilon)
- M (r^{(2)} - r^{(1)})\\
&\geq \mathbf{d}(\gamma^{(2)}).
\end{align*}
We conclude that $\mathbf{d}(\gamma)$ is maximized on the interior of $\Gamma_P$.
\end{proof}

\subsection{Almost exactness in a semi-random setting}
\label{subsec:semi_random}

This section considers semi-random QCQPs of the form
\begin{align}
\label{eq:sr_qcqp}
\inf_{x\in\R^n}\set{q_\obj(x):\, \begin{array}
	{l}
	q_i(x) = 0 ,\,\forall i\in[m]\\
	\norm{x}_2^2\leq 1
\end{array}}.
\end{align}
For notational convenience, define $q_{m+1}(x) \coloneqq \norm{x}_2^2 - 1$.

We will consider the following semi-random model:
First, $A_\obj, A_1,\dots, A_m$ are independently sampled from $\ngoe(n)$. Then, $b_\obj, b_1,\dots,b_m$ and $c_\obj, c_1,\dots,c_m$ are chosen arbitrarily (possibly adversarially depending on the $A_i$s).

Below, we will show that for any fixed $m$, \eqref{eq:sr_qcqp} has an ``almost'' exact SDP relaxation \aas. Specifically, we will apply ideas from \cref{cor:one_shot_obj} to prove:
\begin{proposition}
\label{thm:semi_random}
Fix $\epsilon>0$ and let Let $A_\obj,\dots,A_m\stackrel{\text{i.i.d.}}{\sim} \ngoe(n)$.
Then, \aas,
for all $b_\obj,\dots,b_m\in\R^n$ and $c_\obj,\dots,c_m$, we have
\begin{align*}
\Opt \geq \Opt_\SDP \geq \inf_{x\in\R^n}\set{q_\obj(x) - \epsilon :\, \begin{array}
	{l}
	q_i(x) \in[\pm \epsilon] ,\,\forall i\in[m]\\
	\norm{x}_2^2 \leq 1
\end{array}}.
\end{align*}
\end{proposition}

In a slight departure from previous notation, we will write our dual vector as $(\gamma_\obj,\gamma,\gamma_{m+1})\in\R^{1+m+1}$ where $\gamma_{m+1}\in\R$ corresponds to the constraint $\norm{x}_2^2 \leq 1$. As in \cref{subsec:fully_gaussian}, we will emphasize the main ideas in the proof of \cref{thm:semi_random} and leave the proofs of more technical lemmas to \cref{sec:deferred_proofs_obj_appl}.

The following lemma says that in this random model, $\Gamma$ will again converge to the second-order cone. This lemma follows from \cref{lem:fg_lambda_min_concentration}.
\begin{lemma}
\label{lem:semi_random_Gamma}
Fix $r\geq 0$ and $\epsilon>0$.  Let $A_\obj,A_1\dots, A_m\stackrel{\text{i.i.d.}}{\sim}\ngoe(n)$. Then, \aas,
\begin{align*}
\lambda_{\min}(A(\gamma_\obj, \gamma,1)) \in [1-r\pm\epsilon],~\forall (\gamma_\obj,\gamma)\in r\bS^{m}.
\end{align*}
In particular, \aas,
\begin{align*}
&\set{(\gamma_\obj, \gamma, \gamma_{m+1}):\, \norm{(\gamma_\obj,\gamma)}_2 \leq (1-\epsilon)\gamma_{m+1}}\subseteq
\Gamma\\
&\qquad\subseteq\set{(\gamma_\obj, \gamma, \gamma_{m+1}):\, \norm{(\gamma_\obj,\gamma)}_2 \leq (1+\epsilon)\gamma_{m+1}}.
\end{align*}
\end{lemma}

The following lemma says that a version of \cref{cor:one_shot_obj} with errors holds in this setting. This lemma follows from an $\epsilon$-net argument along with \cref{fact:goe_semicircle}.
\begin{lemma}
\label{lem:semi_random_W}
Fix $\epsilon>0$ and $N\in\N$. Then, \aas,
for every $(\gamma_\obj,\gamma)\in \bS^{m}$, there exists an $N$-dimensional vector space $W\subseteq \R^n$ such that
\begin{align*}
w^\top A(\gamma_\obj,\gamma,1) w \in [\pm \epsilon]\norm{w}_2^2,\,\forall w\in W.
\end{align*}
\end{lemma}

With \cref{lem:semi_random_Gamma,lem:semi_random_W}, we may now prove \cref{thm:semi_random}.
\begin{proof}[Proof of \cref{thm:semi_random}]
Without loss of generality, we assume $\epsilon\in(0,1/2)$ and $b_\obj,b_1,\dots,b_m$, $c_\obj, c_1,\dots,c_m$ are picked so that the SDP relaxation is feasible, i.e.,
\begin{align}
\label{eq:semi_random_sdp_rel}
\infty > \inf_{x\in\R^n}\sup_{(\gamma,\gamma_{m+1})\in\Gamma_P} [(\gamma,\gamma_{m+1}), q(x)].
\end{align}
Let $x^*$ denote an optimizer of \eqref{eq:semi_random_sdp_rel} with value $2t^*$.
Consider the vector $q(x^*) - 2t^*e_\obj \in \R^{1+m+1}$. Without loss of generality, we may assume that $q(x^*) - 2t^*e_\obj$ is both nonzero and on the boundary of $\Gamma^\circ$.
By \cref{lem:semi_random_Gamma} and the assumption that $q(x^*)-2t^*e_\obj\in\bd(\Gamma^\circ)$, we have
\begin{align*}
\tau \coloneqq \sqrt{(q_\obj(x^*) - 2t^*)^2 + \sum_{i=1}^m q_i(x^*)^2} \in [1\pm\epsilon] q_{m+1}(x^*).
\end{align*}
Next, as $q(x^*) - 2t^*e_\obj$ is nonzero, we have that $0 < q_{m+1}(x^*) = 1 - \norm{x^*}^2$, i.e., $\norm{x^*}^2 < 1$.
Hence, by definition of $\tau$, we have $|\tau|\leq 1+\epsilon$.

Set $(f_\obj, f, f_{m+1}) \coloneqq \left(\frac{q_\obj(x^*) - 2t^*}{\tau}, \frac{q_1(x^*)}{\tau},\dots, \frac{q_m(x^*)}{\tau}, 1\right)$ so that $\norm{(f_\obj,f)}_2 = 1$.

Note that by \cref{lem:semi_random_W}, there exists a subspace $W$ of dimension $m+3$ such that
\begin{align*}
w^\top A(f_\obj,f,f_{m+1}) w \in [\pm\epsilon]\norm{w}_2^2 ,~\forall w\in W.
\end{align*}
By a dimension counting argument, there exists a unit $w\in W$ satisfying \begin{align}
\label{eq:w_orthogonal}
\ip{A(\gamma_\obj,\gamma,\gamma_{m+1})x^* + b(\gamma_\obj,\gamma,\gamma_{m+1}), w} = 0
,~\forall (\gamma_\obj,\gamma,\gamma_{m+1}) \in \R^{1+m+1}.
\end{align} 

Then, for this vector $w$ we have
\begin{align}
\label{eq:w_quadratic}
\begin{cases}
	w^\top A(f_\obj, f, 1)w\in[\pm\epsilon],\\
	w^\top A(0, 0_m, 1) w = 1,\text{ and}\\
	w^\top A(\gamma_\obj, \gamma, 1)w \geq 0,~\forall (\gamma_\obj,\gamma)\in (1-\epsilon)\bS^{m}.
\end{cases}
\end{align}
Here, the first two relations follow from $\norm{w}_2^2 = 1$. The third relation follows from \cref{lem:semi_random_Gamma}, which implies that $A(\gamma_\obj,\gamma,1)\succeq 0$ for all $(\gamma_\obj,\gamma)\in(1-\epsilon)\bS^{m}$.

Set $v_\obj \coloneqq w^\top A_\obj w$ and $v\in\R^m$ where $v_i\coloneqq w^\top A_i w$ for $i\in[m]$.
Note that by \eqref{eq:w_orthogonal}, we have
\begin{align*}
q(x^* + \alpha w) - 2t^*e_\obj &= \left(q(x^*) - 2t^*e_\obj\right) + \alpha^2 (v_\obj, v, 1).
\end{align*}
Then, by the first two lines of \eqref{eq:w_quadratic},
\begin{align*}
	&\ip{(v_\obj, v), (f_\obj, f)} = f_\obj w^\top A_\obj w + \sum_{i=1}^m f_i w^\top A_i w\\
	&\qquad = w^\top A(f_\obj,f,1)w - w^\top w \in [-1 \pm \epsilon].
\end{align*}
Next, by the third line of \eqref{eq:w_quadratic}, we have
$\norm{(v_\obj,v)}_2 \leq 1/(1-\epsilon)$.
Set $(\delta_\obj,\delta) \coloneqq (v_\obj,v) + (f_\obj, f)$. We will argue that $(\delta_\obj,\delta)$ is small by bounding its components along $(f_\obj,f)$ and orthogonal to $(f_\obj,f)$,
\begin{align*}
\norm{(\delta_\obj,\delta)}^2 \leq \epsilon^2 + \left(\frac{1}{(1-\epsilon)^2} - (1-\epsilon)^2\right) = O(\epsilon).
\end{align*}

Finally, set $\tilde x = x^* + \alpha w$ where $\alpha = \sqrt{1 - \norm{x^*}^2}$ and note that
\begin{align*}
q(\tilde x) - 2t^* e_\obj &= q(x^*) - 2t^* e_\obj + (1-\norm{x^*}_2^2)\left(v_\obj, v,	1\right)\\
&= q(x^*) - 2t^*e_\obj + (1-\norm{x^*}_2^2)e_{m+1} +
\tau(v_\obj, v, 0)\\
&\qquad + (1-\norm{x^*}_2^2 - \tau)(v_\obj, v, 0)\\
&= q(x^*) - 2t^*e_\obj - (\tau f_\obj, \tau f, \norm{x^*}_2^2 - 1)\\
&\qquad + \tau (\delta_\obj, \delta, 0)
+ (1-\norm{x^*}_2^2 - \tau)(v_\obj, v,0)\\
&= \tau (\delta_\obj, \delta, 0)
+ (1-\norm{x^*}_2^2 - \tau)(v_\obj, v,0).
\end{align*}
The conclusion then follows from the bounds $\abs{\tau} \leq (1+\epsilon)$, $\norm{(\delta_\obj,\delta)}_2 =O(\sqrt{\epsilon})$, $\abs{1- \norm{x^*}_2^2 -\tau}\leq \epsilon$ and $\norm{(v_\obj, v)}_2\leq 1/(1-\epsilon)$.
\end{proof} 
\section*{Acknowledgments}
This research is supported in part by NSF grant CMMI 1454548 and ONR grant N00014-19-1-2321. 

{
\bibliographystyle{plainnat}

}

\begin{appendix}

 \section{Deferred proofs from \cref{sec:conv_applications}}
\label{sec:deferred_proofs_conv_appl}

\subsection{Deferred proofs from \cref{sec:mixed_binary_programming}}

We compute
\begin{align*}
\Gamma &= \set{(\gamma_\obj,\gamma_1,\gamma_2)\in\R_+\times \R^2:\, \begin{pmatrix}
	\gamma_1&\gamma_2/\sqrt{2}\\
	\gamma_2/\sqrt{2}& \gamma_\obj
\end{pmatrix}\succeq 0}\\
&= \set{(\gamma_\obj,\gamma_1,\gamma_2)\in\R^3:\, \begin{array}
	{l}
	\gamma_\obj + \gamma_1 \geq 0\\
	2\gamma_\obj\gamma_1 \geq \gamma_2^2
\end{array}}\\
&=\set{(\gamma_\obj,\gamma_1,\gamma_2)\in\R^3:\, \gamma_\obj + \gamma_1 \geq \sqrt{(\gamma_\obj - \gamma_1)^2  + (\sqrt{2}\gamma_2)^2}}.
\end{align*}
The expression for $\Gamma^\circ$ follows from $\Gamma$.

\begin{proof}
[Proof of \eqref{eq:R_description_mixed_binary}]
Let $(x,t)\in\cS_\SDP\setminus\cS$ such that
$\cG(x,t)$ is a one-dimensional face of $\Gamma^\circ$.
For notational convenience, let $\ell_\obj = q_\obj(x) - 2t$, $\ell_1 = q_1(x)$ and $\ell_2 = q_2(x)$.
Note that $\cG(x,t) = \R_+(\ell_\obj,\ell_1,\ell_2)$
so that $\cF(x,t) = \R_+(-\ell_1,-\ell_\obj,\ell_2)$.
Furthermore, by the assumption that $(\ell_\obj,\ell_1,\ell_2)$ is nonzero and on the boundary of $\Gamma^\circ$, we have
\begin{align*}
\cG(x,t)^\perp = \spann\set{\begin{pmatrix}
	-\ell_1\\
	-\ell_\obj\\
	\ell_2
\end{pmatrix},\, \begin{pmatrix}
	\ell_2\\
	-\ell_2\\
	\ell_1 - \ell_\obj
\end{pmatrix}}.
\end{align*}
We deduce that
\begin{align}
\label{eq:mixed_binary_calc_1}
\cR'(x,t) &= \left\{\begin{pmatrix}
	-\ell_\obj\\
	\ell_2/\sqrt{2}\\
	0
\end{pmatrix},\begin{pmatrix}
	\ell_2/\sqrt{2}\\
	-\ell_1\\
	0
\end{pmatrix},\right.\\
&\qquad\quad \left.\begin{pmatrix}
	-\ell_\obj (2x_1 - 1) + \ell_2 (\sqrt{2}x_2)\\
	-\ell_1 (2x_2) +\sqrt{2} \ell_2(x_1-1)\\
	2\ell_1
\end{pmatrix},
\begin{pmatrix}
	-\ell_2 (2x_1 - 1) + (\ell_1-\ell_\obj) (\sqrt{2}x_2)\\
	\ell_2 (2x_2) +\sqrt{2} (\ell_1-\ell_\obj)(x_1-1)\\
	-2\ell_2
\end{pmatrix}\right\}^\perp.
\end{align}
Here, the first two vectors span $\spann(A(f_\obj,f))$. The second two vectors correspond to the constraints $\ip{A(\gamma_\obj,\gamma)x, x'} - \gamma_\obj t' = 0$ for $(\gamma_\obj,\gamma)\in\cG(x,t)^\perp$.

Below, we will simplify this expression. By the assumption that $(\ell_\obj,\ell_1,\ell_2)$ is nonzero and on the boundary of $\Gamma^\circ$, we have
\begin{align*}
-\ell_\obj - \ell_1 = \sqrt{(\ell_\obj-\ell_1)^2 + (\sqrt{2}\ell_2)^2}
\end{align*}
where the term within the radical is nonzero. Expanding, we deduce that
\begin{align*}
\begin{cases}
	\ell_\obj + \ell_1 < 0\\
	\ell_2^2 - 2\ell_\obj\ell_1 = 0
\end{cases}
=
\begin{cases}
x_2^2 - 2t + x_1(x_1-1) < 0\\
(x_2^2 - 2tx_1)(x_1-1) = 0
\end{cases}.
\end{align*}
Note that $(0,1,0)\in\Gamma$ so that $x_1\in[0,1]$. 
If $x_1 = 1$, then $\ell_\obj <0$, $\ell_1 = 0$ and $\ell_2 = 0$ so that $(x_1,x_2,t)\in\cS$, a contradiction.
We deduce $1-x_1>0$ and $x_2^2 - 2tx_1=0$ and that
$(\ell_\obj,\ell_1,\ell_2) = (x_1-1) (2t, x_1, \sqrt{2}x_2)$. Plugging this into \eqref{eq:mixed_binary_calc_1} gives
\begin{align*}
\cR'(x,t) &= \left\{\begin{pmatrix}
	-2t\\
	x_2\\
	0
\end{pmatrix},\begin{pmatrix}
	x_2\\
	-x_1\\
	0
\end{pmatrix},\begin{pmatrix}
	t\\
	-x_2\\
	x_1
\end{pmatrix},
\begin{pmatrix}
	x_2(x_1-1 + 2t)\\
	- x_1^2+x_1 -2tx_1 - 2t\\
	2x_2
\end{pmatrix}\right\}^\perp.\qedhere
\end{align*}
\end{proof}

\subsection{Deferred proofs from \cref{sec:partition_problem}}
 
We will prove \cref{lem:partition_explicit} in the following series of lemmas. Note that the first identity of \cref{lem:partition_explicit} follows from definition.
To prove the second identity of \cref{lem:partition_explicit}, we will partition $\Gamma_P$ into $n + 1$ pieces depending on the sign pattern of $\gamma\in\Gamma_P$.

Note that $\gamma\in\Gamma_P$ if and only if $aa^\top + \Diag(\gamma)\succeq 0$. In particular, $\gamma\in\Gamma_P$ if $\gamma$ is nonnegative. On the other hand, by the Eigenvalue Interlacing Theorem, $\gamma\notin\Gamma_P$ if it has at least two negative coordinates. It remains to understand $\cN_i \coloneqq \Gamma_P \cap \set{\gamma\in\R^n:\, \gamma_i <0,\,\gamma_i\geq 0,\,\forall j\neq i}$.
The next lemma follows from a straightforward application of the Schur Complement Lemma and the Sherman--Morrison Formula.

\begin{lemma}
Suppose \cref{as:partition_positive_weights} holds. Then, for any $i\in[n]$,
\begin{align*}
\cN_i = \set{\gamma\in\R^n:\, \begin{array}
	{l}
	0>\gamma_i \geq \frac{-a_i^2}{1 + \sum_{j\neq i} a_j^2 / \gamma_j}\\
	\gamma_j >0,\,\forall j\neq i\\
\end{array}}.
\end{align*}
\end{lemma}
\begin{proof}
Without loss of generality we assume $i = n$. For convenience, let $\bar\gamma$ and $\bar a$ denote the first $n - 1$ entries of $\gamma$ and $a$ respectively. By \cref{as:partition_positive_weights}, we have that $\gamma_j>0$ for all $j<n$ (as otherwise the $2\by 2$ minor of $aa^\top + \Diag(\gamma)$ corresponding to $(j,n)$ is not positive semidefinite).
The Schur Complement Lemma and the Sherman--Morrison Formula then imply that $\gamma\in\cN_n$ if and only if $\gamma_n<0$, $\bar\gamma>0$ and
\begin{align*}
\gamma_n + a_n^2 &\geq a_n^2 \bar a^\top (\bar a \bar a^\top + \Diag(\bar \gamma))^{-1} \bar a\\
&= a_n^2 \bar a^\top \left(\Diag(\bar\gamma)^{-1} - \tfrac{\Diag(\bar\gamma)^{-1}\bar a \bar a^\top \Diag(\bar\gamma)^{-1}}{1 + \bar a^\top \Diag(\bar \gamma)^{-1}\bar a}\right) \bar a\\
&= a_n^2\tfrac{\bar a^\top \Diag(\bar\gamma)^{-1}\bar a}{1 + \bar a^\top \Diag(\bar\gamma)^{-1}\bar a}.
\end{align*}
Rearranging terms completes the proof.
\end{proof}

Then decomposing $\Gamma_P = \R^n_+ \cup \bigcup_{i\in[n]}\cN_i$, we get
\begin{align*}
\cS_\SDP &= \set{(x,t)\in\R^n:\, \begin{array}
	{l}
	2t \geq \max_{i\in[n]}\sup_{\gamma\in\cN_i} [\gamma,q(x)]\\
	x\in[\pm1]^n
\end{array}}
\end{align*}

It remains to prove the following lemma.
\begin{lemma}
\label{lem:partition_sup_calculation}
Suppose \cref{as:partition_positive_weights} holds and let $i\in[n]$. For any $x\in[\pm1]^n$, we have
\begin{align}
\label{eq:partition_epigraph_piece}
\sup_{\gamma\in\cN_i}[\gamma,q(x)] = (a^\top x)^2 + \left(a_i \sqrt{1-x_i^2} - \sum_{j \neq i} a_j \sqrt{1 - x_j^2}\right)^2_+.
\end{align}
\end{lemma}

We will need the following two useful facts.
\begin{lemma}
\label{lem:partition_useful}
Let $\xi\in\R^{k}_-$ and $\alpha >0$, then
\begin{align}
\sup_{\zeta\in\R^{k}_{++}}\set{\sum_{i=1}^{k} \frac{\xi_i}{\zeta_i^2 }:\, \sum_{i=1}^k \zeta_i^2 = \alpha} &= -\frac{1}{\alpha}\left(\sum_{i=1}^{k} \sqrt{-\xi_i}\right)^2.
\end{align}
\end{lemma}
\begin{proof}
Without loss of generality, we may assume $\xi\in\R^{k}_{--}$. Then by Cauchy-Schwarz, we have $-\sum_{i=1}^{k} \xi_i/\zeta_i^2 = -\tfrac{1}{\alpha}\left(\sum_{i=1}^{k}\xi_i/\zeta_i^2\right)\left(\sum_{i=1}^{k} \zeta_i^2\right)\geq \tfrac{1}{\alpha}\left(\sum_{i=1}^{k} \sqrt{-\xi_i}\right)^2$. Furthermore, equality holds when $\zeta_i^2\propto\sqrt{\xi_i}$.
\end{proof}

\begin{lemma}
\label{lem:optimizing_over_alpha}
Let $\alpha,\beta \geq 0$, then
\begin{align*}
\sup_{x>0}\left(\frac{\alpha}{1+x} - \frac{\beta}{x}\right) = \left(\sqrt{\alpha} - \sqrt{\beta}\right)_+^2.
\end{align*}
\end{lemma}
\begin{proof}
Let $f(x)\coloneqq \alpha(1+x)^{-1} - \beta x^{-1}$. Note that $\tfrac{d}{dx}f(x) = -\alpha(1+x)^{-2}  + \beta x^{-2}$. There are three cases to consider.
If $\beta = 0$, then $f(x) = \alpha(1+x)^{-1}$ and $\sup_{x>0} \alpha(1+x)^{-1} = \alpha$.
Next, suppose $0\leq \alpha \leq \beta$, then $\tfrac{d}{dx} f(x) = -\alpha(1+x)^{-2} + \beta x^{-2} \geq \beta(x^{-2} - (1+x)^{-2}) \geq 0$ so that $\sup_{x>0}f(x) = \lim_{x\to\infty}f(x) = 0$. Finally, suppose $0 < \beta < \alpha$.
Note that $f'(x)>0$ for all $x$ small enough. Similarly, $f'(x) < 0$ for all $x$ large enough. We deduce that $\sup_{x>0} f(x)$ is achieved. Computing the first-order-necessary conditions, we see that $f(x)$ is maximized at $\tfrac{\sqrt{\beta}}{\sqrt{\alpha} - \sqrt{\beta}}$ with value $\left(\sqrt{\alpha} - \sqrt{\beta}\right)^2$.
\end{proof}

\begin{proof}[Proof of \cref{lem:partition_sup_calculation}]
Without loss of generality, $i = n$.
Let $b \in\R^n_-$ where $b_j = x_j^2-1$. Let $\bar\gamma$ denote the first $n-1$ entries of $\gamma$. Then,
\begin{align*}
\sup_{\gamma\in\cN_n} [\gamma,q(x)]  - q_\obj(x) = \sup_{\gamma\in\cN_n}\ip{\gamma,b}
&= \sup_{\bar\gamma\in\R^{n-1}_{++}} \sum_{i=1}^{n-1}\gamma_ib_i - \frac{a_n^2b_n}{1 + \sum_{i=1}^{n-1} a_i^2/\gamma_i}\\
&= \sup_{\alpha>0}\left(
- \frac{a_n^2b_n}{1 + \alpha}
+
\sup_{\zeta\in\R^{n-1}_{++}}\set{\sum_{i=1}^{n-1} \frac{a_i^2b_i}{\zeta_i}:\, \sum_{i=1}^{n-1} \zeta_i = \alpha}
\right)\\
&= \sup_{\alpha>0}\left(
\frac{\left(a_n\sqrt{-b_n}\right)^2}{1 + \alpha}
- \frac{1}{\alpha}\left(\sum_{i=1}^{n-1} a_i\sqrt{-b_i}\right)^2
\right)\\
&= \left(a_n\sqrt{- b_n} - \sum_{i=1}^{n-1} a_i\sqrt{-b_i}\right)_+^2.
\end{align*}
Here, the second line follows from a change of variables of $\zeta_i \coloneqq a_i^2/\gamma_i$ and $\alpha \coloneqq\sum_{i=1}^{n-1}\zeta_i$. The third line follows from \cref{lem:partition_useful} and the fourth line follows from \cref{lem:optimizing_over_alpha}.
\end{proof}

\begin{proof}
[Proof of \cref{cor:partition_sdp=0}]
Let $\gamma\in\Gamma_P$ and $x\in\R^n$. By convexity of $[\gamma,q(x)]$ in $x$ and the fact that $q(x) = q(-x)$, we deduce that $[\gamma,q(0)]\leq [\gamma,q(x)]$.  We deduce that $\Opt_\SDP=\inf_x\sup_{\gamma\in\Gamma_p}[\gamma,q(x)] = \sup_{\gamma\in\Gamma_P}[\gamma,q(0)]$. By \cref{lem:partition_explicit}, we conclude that
\begin{align*}
\Opt_\SDP  &= \max_{i\in[n]}\left(a_i - \sum_{j\neq i} a_j\right)^2_+.\qedhere
\end{align*}
\end{proof}

\begin{proof}
[Proof of \cref{thm:partition_conv_neq}]
Pick an open set $U\subseteq[\pm1]^n$ such that
\begin{align*}
a_1(1-x_1^2) > \sum_{j>1} a_j(1 - x_j^2),\,\forall x\in U.
\end{align*}
Then by \cref{lem:partition_explicit}, for any $x\in U$, we have $(x,t)\in\cS_\SDP$ if and only if
\begin{align*}
2t \geq f(x) \coloneqq (a^\top x)^2 + \left(a_1\sqrt{1-x_1^2} - \sum_{j>1} a_j \sqrt{1 -x_j^2}\right)^2.
\end{align*}
Note that $f(x)$ is smooth on $U$ and nonlinear (for example note $\frac{\partial^2 f(x)}{\partial x^2}\neq 0$). We conclude that $\cS_\SDP \neq \conv(\cS)$ as $\conv(\cS)$ is polyhedral.
\end{proof} \section{Deferred proofs from \cref{sec:obj_appl}}
\label{sec:deferred_proofs_obj_appl}

\subsection{Useful lemmas}
We first recall that under some minor conditions, pointwise convergence implies uniform convergence for convex functions. We extend this statement to show that pointwise \aas{} convergence implies \aas{} uniform convergence.
\begin{lemma}
\label{lem:convex_pw_implies_uniform}
Let $\Omega\subseteq \R^n$ be an open set and let $f:\Omega\to\R$ be a convex function. Suppose $g_1,g_2,\dots:\Omega\to\R$ is a sequence of random convex functions such that for all $x\in\Omega$ and $\epsilon>0$, we have that \aas,
\begin{align*}
\abs{g_i(x) - f(x)}\leq \epsilon.
\end{align*}
Then, for any compact $C\subseteq \Omega$ and $\epsilon>0$, we have that \aas,
\begin{align*}
\abs{g_i(x) - f(x)}\leq \epsilon,~\forall x\in C.
\end{align*}
\end{lemma}
\begin{proof}
Fix $C\subseteq\Omega$ compact. Without loss of generality, we will assume that $\epsilon>0$ satisfies $C+B(0,3\epsilon)\subseteq \Omega$ and that $f$ is $1$-Lipschitz on $C+B(0,3\epsilon)$.

Fix a finite net $\cN\subseteq C+B(0,3\epsilon)$ such that for all $x\in C + B(0,2\epsilon)$, we have $x\in\conv\left(\cN \cap B(x,\epsilon)\right)$. By our assumption and the fact that $\cN$ is finite, we have that \aas, $\abs{f(x) - g_i(x)}\leq \epsilon$ for all $x\in\cN$. We condition on this event in the remainder of the proof.

For any $x\in C$, let $x = \sum_j \lambda_jx_j$ denote the convex decomposition guaranteed by $x\in\conv\left(\cN \cap B(x,\epsilon)\right)$. Then,
\begin{align*}
g_i(x) \leq \sum_j \lambda_j g_i(x_j) \leq \sum_j\lambda_j \left(f(x_j) + \epsilon\right) \leq f(x) + 2\epsilon.
\end{align*}
Here, the last inequality follows from $f(x_j)\leq f(x) + \norm{x - x_j}_2 \leq f(x) + \epsilon$.

Let $x\in C$ and $x' \in \cN \cap B(x,\epsilon)$. Note that $y \coloneqq x' + (x' - x) \in C + B(0,2\epsilon)$. By construction, there exists $y' \in\cN\cap B(y,\epsilon)$ such that $g_i(y')\geq g_i(y)$. Finally,
\begin{align*}
f(x) + 4\epsilon &\geq f(y') + \epsilon \geq g_i(y') \geq g_i(y) \geq 2g_i(x') - g_i(x) \geq 2(f(x') - \epsilon) - g_i(x) \geq 2 f(x)-g_i(x) - 4\epsilon.
\end{align*}
Therefore, by rearranging and combining, we deduce that \aas, $\abs{g_i(x) - f(x)}\leq 8\epsilon,\,\forall x\in C$.
\end{proof}

\begin{lemma}
\label{lem:int_is_semicircle}
Let $r\in[-1,1]$, then
\begin{align*}
-\int_{\sigma = -1}^1 \frac{r^2}{1 + r \sigma}\,d\mu_{\textup{nsc}}(\sigma) =2(\sqrt{1 - r^2} - 1)= \phi(r).
\end{align*}
\end{lemma}
\begin{proof}
We begin by expanding the definition of $\mu_{\textup{nsc}}$ and substituting $\sigma = -\cos \theta$:
\begin{align}
-\int_{\sigma = -1}^1 \frac{r^2}{1 + r \sigma} \,d\mu_{\textup{nsc}}(\sigma) &= -\frac{2}{\pi}\int_{\sigma = -1}^1 \frac{r^2 \sqrt{1 - \sigma^2}}{1 + r\sigma}\,d\sigma \nonumber\\
&= -\frac{2}{\pi}\int_{\theta= 0}^\pi \frac{r^2\sin^2\theta}{1 - r\cos\theta}\,d\theta \nonumber\\
&= -\frac{2}{\pi}\int_{\theta= 0}^\pi \frac{r^2 - r^2\cos^2\theta}{1 - r\cos\theta}\,d\theta \nonumber\\
&= -\frac{2}{\pi}\int_{\theta= 0}^\pi \frac{r^2 -1}{1 - r\cos\theta}\,d\theta -\frac{2}{\pi}\int_{\theta= 0}^\pi (1 + r\cos\theta) \,d\theta \nonumber\\
&= \frac{2(1 - r^2)}{\pi}\left(\int_{\theta = 0}^\pi \frac{1}{1 - r\cos\theta}\,d\theta\right) - 2. \label{eq:lem_circ1}
\end{align}
We now focus on the bracketed integral. Perform the change of variables $\theta = 2\eta$ to get
\begin{align}
\label{eq:lem_circ2}
\int_{\theta = 0}^\pi \frac{1}{1 - r\cos\theta}\,d\theta &= 2\int_{\eta = 0}^{\pi/2} \frac{1}{1 - r\cos(2\eta)}\,d\eta.
\end{align}
Recalling the identities $\cos(2\eta)=2\cos^2(\eta)-1$ and $\cos^{-2}\eta= \sec^2\eta = \tan^2\eta + 1 = \tfrac{d}{d\eta}\tan(\eta)$, we then have
\begin{align*}
\frac{1}{1 - r \cos(2\eta)} = \frac{1}{1 +r -2r \cos^2\eta} = \frac{\tfrac{d}{d\eta}\tan\eta}{(1+r)\tan^2 \eta + (1 - r)}.
\end{align*}
Performing one last change of variables $t = \tan\eta$ gives
\begin{align}
2\int_{\eta = 0}^{\pi/2} \frac{1}{1 - r\cos(2\eta)}\,d\eta 
&= 2\int_{\eta = 0}^{\pi/2} \frac{\tfrac{d}{d\eta}\tan\eta}{(1+r)\tan^2 \eta + (1 - r)}\,d\eta \nonumber\\
&= 2\int_{t = 0}^\infty \frac{1}{(1+r)t^2 + (1-r)}\,dt \nonumber\\
&= 2 \left.\frac{\arctan\left(t\sqrt{\tfrac{1+r}{1-r}}\right)}{\sqrt{1-r^2}}\right|_{t = 0}^\infty \nonumber\\
&= \frac{\pi}{\sqrt{1-r^2}}. \label{eq:lem_circ3}
\end{align}
Combining \eqref{eq:lem_circ1}, \eqref{eq:lem_circ2}, and \eqref{eq:lem_circ3} gives the desired identity.
\end{proof}

\subsection{Deferred proofs from \cref{subsec:fully_gaussian}}
\begin{proof}
[Proof of \cref{lem:fg_lambda_min_concentration}]
Let $\Omega=\R^m$ and set $f(\gamma) \coloneqq 1 - \norm{\gamma}_2$. Note that $f$ and $\lambda_{\min}\left(A[\gamma]\right)$ are both concave functions on $\Omega$.
We have $\lambda_{\min}(A[0]) = 1 = f(0)$. 
Furthermore, for any nonzero $\gamma\in\R^m$ and $\epsilon>0$, 
\begin{align*}
\lambda_{\min}\left(A[\gamma]\right) &= 1 + \norm{\gamma}_2\lambda_{\min}\left(\sum_{i=1}^m \frac{\gamma_i}{\norm{\gamma}_2} A_i\right) \in 1 + \norm{\gamma}_2[-1\pm\epsilon] =  [f(\gamma) \pm \norm{\gamma}_2\epsilon],\,\aas.
\end{align*}
Here, the inclusion holds by \cref{fact:goe_reweighting_invariant,fact:goe_bounded_norm}.
Taking $C = r \bS^{m-1}$ and applying \cref{lem:convex_pw_implies_uniform}.
\end{proof}

\begin{proof}
[Proof of \cref{lem:fg_sphere_cap_concentration}]
Fix $r\in(0,1)$. Without loss of generality, $r+2\epsilon <1$.
Set $\Omega \coloneqq \set{\gamma\in\R^m:\, \norm{\gamma}_2< r + 2\epsilon}$.
Let $\hat\gamma\in\Omega$.
Note that we may generate $A[\hat\gamma]$ and $b[\hat\gamma]$ via the following equivalent process: Sample $\bar A \sim \ngoe(n)$ and $\bar b \sim N(0,I_n/n)$ independently and set $A[\hat\gamma] \coloneqq I + r \bar A$ and $b[\hat\gamma]\coloneqq r\bar b$. With this notation, $-b[\hat\gamma]A[\hat\gamma]^{-1}b[\hat\gamma] = -r^2 \bar b^\top (I + r \bar A)^{-1} \bar b$. 
Let $\bar A = \sum_{i=1}^n \sigma_i v_iv_i^\top$ be the eigenvalue decomposition of $\bar A$ and let $\mu_{\bar A}$ denote its Empirical Spectral Distribution. By \cref{lem:int_is_semicircle}, we have
\begin{align*}
&\frac{1}{r^2}\abs{-b\left[\hat\gamma\right]^\top A\left[\hat\gamma\right]^{-1}b\left[\hat\gamma\right] - \phi(r)}\\
&\qquad =\frac{1}{r^2}\abs{-b\left[\hat\gamma\right]^\top A\left[\hat\gamma\right]^{-1}b\left[\hat\gamma\right] + \int_{\sigma = -1}^1 \frac{r^2}{1+r\sigma}d\mu_{\textup{nsc}}}\\
&\qquad = 
\abs{\bar b^\top (I + r\bar A)^{-1} \bar b - \int_{\sigma = -1}^1 \frac{1}{1+r\sigma}d\mu_{\textup{nsc}}}\\
&\qquad \leq
\abs{\sum_{i=1}^n\frac{\left(v_i^\top \bar b\right)^2 - 1/n}{1+r\sigma_i}} +
\abs{\int \frac{1}{1+r\sigma} d\mu_{\bar A}(\sigma) - \int\frac{1}{1+r\sigma}d\mu_{\textup{nsc}}(\sigma)},
\end{align*}
where the last inequality follows from the identity $(I+r\bar A)^{-1}= \sum_{i=1}^n {1\over 1+r \sigma_i} v_iv_i^\top$ and Cauchy-Schwartz inequality. 
Note that by \cref{fact:goe_bounded_norm}, for all $i\in[n]$ we have that $1+r\sigma_i \geq 1 - r - r\epsilon \geq 1 - r - \epsilon > \epsilon$ \aas.
We will compute the mean and variance of the first term conditioned on this event.
By independence of $\bar b$ and $\bar A$,
\begin{align*}
\E_{\bar b}\left[\sum_{i=1}^n\frac{\left(v_i^\top \bar b\right)^2 - 1/n}{1+r\sigma_i}\,\middle |\, 1+r\sigma_i \geq \epsilon,\,\forall i \right] &=
\sum_{i=1}^n\left(\frac{1}{1+r\sigma_i}\right)\E_{\bar b}\left[\left(v_i^\top \bar b\right)^2 - \frac{1}{n}\,\middle |\, 1+r\sigma_i \geq \epsilon,\,\forall i\right] = 0,\text{ and}\\
\E_{\bar b}\left[\left(\sum_{i=1}^n\frac{\left(v_i^\top \bar b\right)^2 - 1/n}{1+r\sigma_i}\right)^2\,\middle |\, 1+r\sigma_i \geq \epsilon,\,\forall i\right] &\leq \left(\frac{1}{\epsilon}\right)
\E_{\bar b}\left[\left(\sum_{i=1}^n\left(v_i^\top \bar b\right)^2 - 1/n\right)^2\,\middle |\, 1+r\sigma_i \geq \epsilon,\,\forall i\right] = \frac{2}{\epsilon n}.
\end{align*}
In particular, the first term can be bounded by $\epsilon/(2r^2)$ \aas.

For the second term, define the $C_c^\infty$ function
\begin{align*}
\psi(x) \coloneqq \begin{cases}
	\tfrac{1}{1+r x}, &\text{if } \abs{x}\leq 1+\delta\\
	0, & \text{if } \abs{x} \geq 1+2\delta\\
	C_c^\infty, & \text{else}.
\end{cases}
\end{align*}
By \cref{fact:goe_bounded_norm}, we have that \aas{} $\int \frac{1}{1+r\sigma}d\mu_{\bar A}(\sigma) = \int \psi(\sigma) d\mu_{\bar A}(\sigma)$. Applying \cref{fact:goe_semicircle}, we conclude that the second term can be bounded by $\epsilon/(2r^2)$ \aas.

Combining the two bounds shows that for any $\gamma\in\Omega$ and $\epsilon>0$, $\abs{-b[\gamma]^\top A[\gamma]^{-1}b[\gamma] - \phi(\gamma)}\leq \epsilon$ \aas. Applying \cref{lem:convex_pw_implies_uniform} with $C = r\bS^{m-1}$ concludes the proof.
\end{proof}

\subsection{Deferred proofs from \cref{subsec:semi_random}}
\begin{lemma}
\label{lem:W_single}
Fix $\epsilon>0$ and $N\in\N$. Let $A\sim\ngoe(n)$. Then, \aas{} there exists a $N$-dimensional vector space $W\subseteq\R^n$ such that
\begin{align*}
w^\top A w \in [1\pm \epsilon]\norm{w}^2,\,\forall w\in W.
\end{align*}
\end{lemma}
\begin{proof}
Let $\psi$ denote a $C_c^\infty$ function from $\R$ to $[0,1]$ that takes the value one on $[1\pm \epsilon/2]$ and the value zero outside of $[1\pm \epsilon]$.
Note that $\theta \coloneqq \int \psi d\mu_\nsc$ is some positive constant independent of $n$.
Let $W$ denote the vector space corresponding to the eigenvalues of $A$ in the range $[1\pm \epsilon]$. Clearly $w^\top A w \in[1\pm\epsilon]\norm{w}_2^2$ for all $w\in W$. It remains to note that by \cref{fact:goe_semicircle}, we have \aas
\begin{align*}
\frac{\dim(W)}{n} = \frac{\abs{\set{i\in[n]:\, \lambda_i(A)\in[1\pm \epsilon] }}}{n} &\geq \int \psi d\mu_{\bar A} \geq \int \psi d\mu_{\nsc} - \theta/2 = \theta/2
\end{align*}
so that $\dim(W) \geq N$ \aas.
\end{proof}

\begin{proof}
[Proof of \cref{lem:semi_random_W}]
Let $\cN$ denote a finite $\epsilon$-net on $\bS^m\subseteq\R^{1+m}$. By \cref{lem:W_single}, \aas, for every $(\gamma_\obj,\gamma)\in\cN$, there exists an $N$ dimensional subspace $W$ such that
\begin{align*}
w^\top A(\gamma_\obj, \gamma, 1) w \in [\pm\epsilon]\norm{w}^2,\,\forall w\in W.
\end{align*}
Furthermore, by \cref{lem:fg_lambda_min_concentration}, we have that \aas{} $\norm{A(\gamma_\obj,\gamma,0)}_2 \in \norm{(\gamma_\obj,\gamma)}[1\pm \epsilon]$ for all $(\gamma_\obj,\gamma)\in\R^m$. 
We condition on these two events.

Now, let $(\gamma_\obj,\gamma)\in\bS^{m}$ and let $(\gamma_\obj',\gamma')\in\cN \cap B((\gamma_\obj,\gamma),\epsilon)$. Let $W$ denote the $N$-dimensional subspace guaranteed for $(\gamma_\obj',\gamma')$.
Then for all $w\in W$,
\begin{align*}
w^\top A(\gamma_\obj,\gamma,1) w &= w^\top A(\gamma'_\obj,\gamma',1) w + w^\top A(\gamma_\obj - \gamma'_\obj, \gamma - \gamma')w \in [\pm 3\epsilon]\norm{w}^2.\qedhere
\end{align*}
\end{proof}

 \end{appendix}

\end{document}